\documentclass[ijoc,sglanonrev]{informs5}
\RequirePackage{tgtermes}
\RequirePackage{newtxtext}
\RequirePackage{newtxmath}
\RequirePackage{bm}
\RequirePackage{endnotes}

\OneAndAHalfSpacedXII 

\usepackage[ruled,linesnumbered]{algorithm2e}
\usepackage{tikz}

\usepackage{natbib}
 \bibpunct[, ]{(}{)}{,}{a}{}{,}%
 \def\bibfont{\small}%
 \def\BIBand{and}%

\usepackage{booktabs}
\usepackage{makecell}
\usepackage{hyperref}
\usepackage[para]{threeparttable}
\usepackage{mathtools}
\usepackage{bbm}
\usepackage{enumitem}
\usepackage{multirow}
\usepackage[caption=false]{subfig}

\EquationsNumberedThrough    

\TheoremsNumberedThrough     
\ECRepeatTheorems  %

\MANUSCRIPTNO{IJOC-0001-2026.00}

\DeclarePairedDelimiterX\Set[2]{\lbrace}{\rbrace}{ #1 \,\delimsize| \,\mathopen{} #2 }

\newcommand{\ep}[2]{\mathbb{E}_{#1} \left[ #2 \right]}

\newcommand{\bs}[1]{\boldsymbol{#1}} 
\newcommand{\Bs}[1]{\mathbb{#1}} 
\newcommand{\Cs}[1]{\mathcal{#1}} 
\newcommand{\Fs}[1]{\mathfrak{#1}} 

\newcommand{\ul}[1]{\underline{#1}}
\newcommand{\ol}[1]{\overline{#1}}

\newcommand{\conv}[1]{\text{conv}\left(#1\right)}

\newcommand{\st}{\text{s.t.}}

\begin{document}


\RUNAUTHOR{Rahimian and Mehrotra}

\RUNTITLE{Incorporating Decision Dependence in Distributional Ambiguity}

\TITLE{On the Relevance of Incorporating Decision Dependence in Distributional Ambiguity}
\ARTICLEAUTHORS{%
\AUTHOR{Hamed Rahimian}
\AFF{Department of Industrial Engineering, Clemson University, Clemson SC 29634, USA  \EMAIL{hrahimi@clemson.edu}}

\AUTHOR{Sanjay Mehrotra}
\AFF{Department of Industrial Engineering and Management Sciences, Northwestern University, Evanston IL 60208, USA  
\EMAIL{mehrotra@northwestern.edu}}
} 

\ABSTRACT{%
Most existing studies on distributionally robust optimization (DRO) with a decision-dependent ambiguity set focus on the computational and theoretical challenges posed by this class of problems. In this paper, we adopt a combined modeling and computational perspective to understand the trade-offs between modeling fidelity, solution quality, and computational effort, particularly in comparison with DRO models using decision-independent ambiguity sets.
Motivated by representative applications in joint pricing-stocking newsvendor problems with price-dependent demand and facility location problems with location-dependent demand, we consider a two-stage stochastic mixed-integer program with (non)convex continuous recourse. 
Assuming a finite sample space, we model the decision-dependent distributional ambiguity with a polyhedral ambiguity set and reformulate the problem as a nonconvex mixed-integer nonlinear program. 
To efficiently solve the reformulations, we propose 
decomposition-based cutting-plane algorithms. 
Our experiments on benchmark instances indicate that a decision-dependent ambiguity set substantially reduces the postdecision disappointment for the newsvendor and out-of-sample cost for the facility location problems, up to 65\% and 7\% on average, respectively;   
thereby mitigating the optimizer's curse relative to a decision-independent DRO. However, this improvement is achieved at the expense of increased computational effort, with the absolute runtime ratio falling between 2 and 25 for medium- to large-sized newsvendor instances and between 2 and 12 for facility location instances. 


}%

\FUNDING{The authors gratefully acknowledge the support of the Office of Naval Research through grant N00014-18-1-2097-P00001. The first author gratefully acknowledges the support of the U.S. Air Force Office of Scientific Research through grant FA9550-24-1-0241.}



\maketitle



\section{Introduction}

Distributionally robust optimization (DRO) \citep{rahimian2022frameworks} has attracted significant attention as a means of balancing the conservatism inherent in robust optimization (RO) 
against the optimizer's curse associated with stochastic programming (SP). 
Despite substantial theoretical and algorithmic advances in DRO, most research often assumes (i) uncertain parameters (or their distributions) are known a priori or belong to a set and (ii) are {\it exogenous}. 
However, these assumptions rarely hold in practice, and 
uncertain parameters may be endogenously/adapatively affected by the decisions \citep{luo2020distributionally}.

To capture this interplay, a growing line of research has studied the framework of {\it decision-dependent} DRO. 
Recent studies in this area have primarily focused on addressing the computational and theoretical challenges posed by this class of problems, seeking to deepen understanding of their structure and solution methodologies. However, several critical questions remain largely unanswered in the literature. Is decision-dependency merely a modeling preference, or is it an essential requirement? Can one safely ignore decision-dependency and still obtain high-quality solutions, or does such misspecification inevitably lead to substantively suboptimal decisions? And, importantly, what is the trade-off between computational burden and solution quality?

In this paper, perhaps for the first time, we investigate decision-dependent DRO models from a combined modeling and computational perspective. Through a comparison with standard DROs, we aim to elucidate how the incorporation of adaptive ambiguity influences the performance of resulting solutions, as well as the computational burden to obtain them.

\subsection{Representative Applications}
To reach our goals, we investigate two representative applications:
(i) a joint pricing-stocking problem in a multiproduct newsvendor setting with price-dependent demand and
(ii) a facility location problem with location-dependent demand.

\subsubsection{Multiproduct Newsvendor Problem}
\label{sec: NV_app}
A substantial body of research on stochastic pricing–stocking problems highlights the importance of understanding how pricing and inventory decisions interact under demand uncertainty, see, e.g., \citet{petruzzi1999pricing}.
Collectively, this literature shows that in many newsvendor-type problems, demand is endogenously influenced by pricing decisions, motivating models that explicitly incorporate price-dependent demand when determining order quantities and price.
Operational models typically embed price elasticity of demand into the classical stochastic newsvendor problem by modeling demand as additive or multiplicative functions of price, see, e.g., \citet{lee2012newsvendor,kocabiyikouglu2011elasticity} and references therein, or by using power- or logit-based market share formulations grounded in modern choice theory \citep{mcfadden1981econometric} and attraction models \citep{huff1963probabilistic,huff1964defining}.

We study stocking-pricing decisions for a multiproduct newsvendor problem under price-dependent demand, where the distributional ambiguity is inspired by an attraction model that accounts for price levels and substitution effects. 
Let $n$ denote the number of products.
For each $i \in [n]$, suppose that $c_{i}$ denotes the per unit purchasing cost, $g_{i}$ denotes the per unit salvage price, and $b_{i}$ denotes the per unit back-order cost. 
For a fixed order quantity $q_{i} \in \Bs{R}$, a fixed per unit selling price $r_{i} \in \Bs{R}$, and demand $\xi_{i} \in \Bs{R}$, cost function $h_{i}(q_{i}, r_i, \xi_{i})$ is defined as: 
\begin{align*}
    h_{i}(q_{i}, r_i, \xi_{i}) & =c_{i}q_{i}-r_{i}\min\{q_{i}, \xi_{i}\} -g_{i}(q_{i}-\xi_{i})_{+}+b_{i}(\xi_{i}-q_{i})_{+} \\
    & = (c_{i} - r_{i})q_{i} + (r_{i}-g_{i})(q_{i}-\xi_{i})_{+}+b_{i}(\xi_{i}-q_{i})_{+},
\end{align*}
where $(a)_{+}$ denotes $\max\{0, a\}$. 
Equivalently, we have 
\begin{equation}
    \label{eq: NV_recourse}
         h_{i}(q_{i}, r_i, \xi_{i}) \! = (c_{i} - r_{i})q_{i} + \! \min_{y_{i}^+, y_i^- \ge 0} \{    (r_{i}-g_{i}) y_i^+ + b_{i} y_i^-  : y_i^+ - y_i^- = q_i - \xi_i\}. 
\end{equation}
Let $h(\bs{q}, \bs{r}, \bs{\xi})= \sum_{i \in [n]} h_{i}(q_{i}, r_i, \xi_{i})$. 
Moreover, define the mixed-integer set $\Cs{X}=\{(\bs{q},\bs{r}): \sum_{i \in [n]} c_i q_i \le d, \; q_i \ge 0, \; q_i \in \Bs{Z}, \; \underline{r}_i \le r_i \le \overline{r}_i, \; i \in [n]\}$, for some $d, \underline{r}_i, \overline{r}_i \in \Bs{R}$, $i \in [n]$. 
Suppose that $\bs{\xi}$ has a finite support $\Xi=\{\bs{\xi}_{\omega}\}_{\omega=1}^{N}$, where $N$ is the number of scenarios.
We then formulate a decision-dependent DRO problem $\min\limits_{(\bs{q},\bs{r}) \in \Cs{X} } \ \max\limits_{\bs{p} \in \Cs{P}(\bs{r})} \ 	\ep{\bs{p}}{h(\bs{q},\bs{r}, \bs{\xi})}$. The price-dependent ambiguity set $\Cs{P}(\bs{r})$ is defined as: 
\begin{equation}
\label{eq: moment_NV}
    \Cs{P}(\bs{r}) = 
    \left\lbrace \bs{p} \ge \bs{0} \middle\vert
    \begin{array}{l}
         (1- \tau_\mu) \bs{\mu}_{0}(\bs{r}) \le  \sum\limits_{\omega \in [N]} p_{\omega} \bs{\xi}_{\omega}    \le (1+\tau_\mu)\bs{\mu}_{0}(\bs{r}), \\  
        \underline{\tau}_\sigma \big(\bs{\sigma}_{0}(\bs{r})^2 + \bs{\mu}_{0}(\bs{r})^2\big) \le \sum\limits_{\omega \in [N]} p_{\omega} \bs{\xi}_{\omega}^{2}\le \overline{\tau}_\sigma \big(\bs{\sigma}_{0}(\bs{r})^2 + \bs{\mu}_{0}(\bs{r})^2\big), \\
        \sum\limits_{\omega \in [N]} p_{\omega} = 1
    \end{array}
    \right\rbrace,
\end{equation}
where $\bs{\mu}_{0}(\bs{r})$ and $\bs{\sigma}_{0}(\bs{r})$ denote the vector of (nominal) price-dependent mean and standard deviation of the random demand $\bs{\xi}$. Moreover, parameters $\tau_\mu$, $\underline{\tau}_\sigma$, and $\overline{\tau}_\sigma$ control the conservatism of resulting robust decisions by adjusting the maximum allowable deviations on the nominal first- and second-order moments, with $0\le \underline{\tau}_\sigma \le 1 \le \overline{\tau}_\sigma$. 
We let  
\begin{equation}
    \mu_{0,j}(\bs{r}) = \overline{\mu}_{j} \big( 1 + \sum_{i \in [n]} u_{i j}^{\mu} r_i \big), 
    \quad 
    \sigma_{0,j}(\bs{r})^2 = \overline{\sigma}_{j}^2 \big( 1 - \sum_{i \in [n]} u_{i j}^{\sigma} r_i \big), 
    \label{NV_moments} \\
\end{equation}
for $j \in [n]$. Here, $\overline{\bs{\mu}}$ and $\overline{\bs{\sigma}}$ indicate the vector of empirical mean and standard deviation of the random demand $\bs{\xi}$. We assume that $u_{i i}^{\mu} < 0 $, $i \in [n]$, implying that an increase in the price of product $i$ leads to a decrease in the average demand for product $i$, $\xi_i$.  Moreover, for $i \in [n]$, nonnegative $u_{i j}^{\mu}$ and $u_{i j}^{\sigma}$, $ i \neq j \in [n]$, capture the impact of other products' price on the mean and standard deviation of the demand for product $i$, $\xi_i$; emphasizing that the products may be substitutable. 
In our numerical experiments in Section \ref{sec: numerical}, these impact parameters are a function of {\it similarity} between products. 
When $u_{i j}^{\mu}=u_{i j}^{\sigma}=0$ for all $i,j \in [n]$, the constraints yield a price-independent ambiguity set. 

\subsubsection{Facility Location Problem}
\label{sec: FL_app}
When new services enter a market, the placement of service facilities can directly influence customers' willingness to adopt the service. 
Empirical studies indicate that customer behavior 
is influenced by accessibility (e.g., travel distance) and service attractiveness (e.g., size, variety) of competing facilities, all of which affect both demand volume and its distribution. 
This phenomenon is well-documented in modern discrete choice theory \citep{mcfadden1981econometric} and in classical location science. Most notably, Huff's probabilistic attraction model \citep{huff1963probabilistic,huff1964defining}, formalizes how the perceived utility of a facility is governed by its attractiveness, the relative attractiveness of alternatives, and a distance-decay effect, extending the classical Reilly's law of retail gravitation \citep{reilly1931retail}. These studies highlight that, in many facility location problems, demand endogenously depends on facility placement, motivating models that explicitly incorporate this dependence.

We study a facility location problem under random location-dependent demand, where the distributional ambiguity is inspired by an attraction model that accounts for distance from open facilities. Consider $n$ potential locations for installing facilities and $m$ customers. For each $i \in [n]$, let $C_i$ denote the capacity. 
For each $j \in [m]$, let $r_j$ and $g_j$ denote the unit revenue and penalty cost for unsatisfied demand, respectively. Moreover, for $i \in [n]$ and $j \in [m]$, let $c_{ij}$ denote the unit transportation cost from location $i$ to customer $j$. 
For a fixed location vector $\bs{x} \in \{0,1\}^n$ and demand $\xi_j \in \Bs{R}$, cost function $h_j(\bs{x}, \xi_{j})$ is defined as in \citep{basciftci2021distributionally}: 
\begin{equation}
    \label{eq: FL_recourse}
    \begin{aligned}
         h_j(\bs{x}, \xi_{j}) = g_j  \xi_j + \min \Big\{ \sum_{i \in [n]}(c_{ij}-r_j-g_j) y_{ij} 
        :\sum_{i \in [n]} y_{ij} \le  \xi_j, \;
        y_{ij} \leq C_i x_i\;
        y_{ij} \geq 0, \ i \in [n]\Big\}, 
    \end{aligned}
\end{equation}
where $y_{ij}$ denote the transported units from location $i \in [n]$ to customer $j \in [m]$. 
Let $h(\bs{x}, \bs{\xi})= \sum_{j \in [m]} h_{j}(\bs{x}, \xi_{j})$. 
Moreover, define the binary set $\Cs{X}=\{\bs{x} \in \{0,1\}^n : \sum_{i \in [n]} x_i = d\}$, for some $d \in \{0\} \cup [n]$. 
Suppose that $\bs{\xi}$ has a finite support $\Xi=\{\bs{\xi}_{\omega}\}_{\omega=1}^{N}$, where $N$ is the number of scenarios.
We then formulate a decision-dependent DRO problem $\min\limits_{\bs{x} \in \Cs{X} } \ \max\limits_{\bs{p} \in \Cs{P}(\bs{x}) } \ 	\ep{\bs{p}}{h(\bs{x},\bs{\xi})}$. The location-dependent ambiguity set $\Cs{P}(\bs{x})$ is defined similarly to \eqref{eq: moment_NV}, with $\bs{x}$ replacing $\bs{r}$. 
We let  
\begin{equation}
    \mu_{0,j}(\bs{x}) = \overline{\mu}_{j} \big( 1 + \sum_{i \in [n]} u_{i j}^{\mu} x_i \big),  \quad 
    \sigma_{0,j}(\bs{x})^2 = \overline{\sigma}_{j}^2  \big( 1 - \sum_{i \in [n]} u_{i j}^{\sigma} x_i \big), \label{FL_moments} 
\end{equation}
for $i \in [n]$, with 
$\overline{\bs{\mu}}$ and $\overline{\bs{\sigma}}$ 
defined as before.
For $j \in [m]$, nonnegative $u_{i j}^{\mu}$ and $u_{i j}^{\sigma}$, $i \in [n]$, capture the impact of opening a facility at location $i \in [n]$ on the mean and standard deviation of the demand at customer $j$, $\xi_j$. In our numerical experiments in Section \ref{sec: FL}, these impact parameters are a function of distance. 
When $u_{i j}^{\mu}=u_{i j}^{\sigma}=0$ for all $i \in [n]$ and $j \in [m]$, the constraints yield a location-independent ambiguity set. 

\subsection{Contributions}

To understand the intrinsic balance between modeling richness, solution quality, and computational efficiency that characterizes the decision-dependent DRO, we consider a  two-stage stochastic mixed-integer program (MIP) with (non)convex continuous recourse.  Assuming a finite sample space, the distributional ambiguity is modeled using a general polyhedral set. 
Here is a list of our contributions. 

\begin{itemize}[leftmargin=*]


    \item We reformulate the resulting DRO as a nonconvex two-stage stochastic mixed-integer nonlinear program (MINLP). 
    To efficiently solve the nonconvex two-stage reformulation, we develop a decomposition-based cutting-plane algorithm using disjunctive cuts derived from the Lagrangian dual of the recourse and establish its finite convergence. 

\item The disjunctive cutting-plane algorithm may be of independent interest, as it applies to solve a  general two-stage SP with nonconvex recourse when the cost vector---or equivalently, the recourse matrix---in the second-stage problem is simultaneously (linearly) bi-parameterized by the mixed-integer first-stage decision variables and uncertain parameters    (see Remark \ref{rem: two-stage}).   

    \item Our experiments on benchmark instances indicate that solving a decision-dependent DRO problem is generally more challenging than its decision-independent counterpart.
    The absolute runtime ratio falls between 2 and 25 for the medium- to large-sized newsvendor instances, whereas it falls between 2 and 12 for the facility location instances. Nevertheless, decision-dependent ambiguity sets produce solutions with superior out-of-sample performance (up to 11\% and 7\% on average, respectively, for the newsvendor and facility location problems), 
    while they tend to exhibit higher in-sample costs. In other words, incorporating decision-dependent ambiguity mitigates the optimizer's curse relative to decision-independent DRO, up to 65\%. 
\end{itemize}


\subsection{Literature Review}
\subsubsection{Decision-Dependent Uncertainty} There is growing interest in decision problems with endogenous uncertainty in the context of SP \citep{dentcheva2020risk,drusvyatskiy2023stochastic,bazotte2026,pantuso2025shaped}, RO \citep{chen2023robust,vayanos2020robust}, and DRO  \citep{noyan2018distributionally}, with applications including pricing \citep{hu2019data,qu2025decision} 
and  
facility location \citep{basciftci2021distributionally,luo2025two}. 
Following the earlier work of \citep{dupacova2006optimization}, two distinct classes of endogenous uncertainty were identified in \citep{hellemo2018decision}, depending on whether the decisions affect the {\it temporal revelation of uncertainty or its probability distribution}---with the possibility of affecting the set of possible outcomes in both classes \citep{nohadani2018optimization}. 
In this work, we are interested in the impact of decisions on probability distributions; hence, we limit our literature review to this class of problems.  

Unlike models with temporal endogenous uncertainty, the literature on models where decisions affect the probability distribution is sparse. 
Various discrepancy-based DRO  approaches were developed in  \citep{luo2020distributionally}, where the maximum allowed distance from a nominal distribution is assumed to be decision-dependent. \citet{noyan2018distributionally,qu2025decision} considered a DRO  approach, where the decision-dependent distributional ambiguity is modeled via Wasserstein distance around a decision-dependent nominal distribution. Similarly,  \cite{basciftci2021distributionally} studied a DRO approach to a facility location problem, where the distributional ambiguity of decision-dependent demand is captured with moments around a decision-dependent nominal moment. \cite{luo2019service} studied a service center location problem, where utility gains upon receiving service are location-dependent and assumed to be ambiguously described by moment-based sets. 
A joint stocking and pricing problem for a product without knowing the price-dependent demand was studied in \cite{hu2019data}. The authors introduced a functionally robust approach to hedge against various classes of decreasing convex or concave functions to model price-dependent demand. 
A DRO problem with a decision-dependent ambiguity set was studied in \cite{royset2017}, where the preference relationship between probability distributions is formed via decision-dependent cumulative distribution functions. 
\citet{qu2025decision} obtained a tractable reformulation of the studied decision-dependent DRO and established finite-sample out-of-sample and optimality guarantee with respect to the true decision-dependent SP. 
A similar problem to the one in \citet{basciftci2021distributionally}  was studied in a multistage setting in \citep{yu2022multistage}. 


\subsubsection{Two-Stage SPs with Nonconvex Recourse}
Two defining characteristics have largely shaped the research on two-stage  SPs. First, the first-stage decision variables typically affect only the constraints---and only linearly---of the second-stage problem, without entering its objective function. Second, the recourse function is generally defined as the value function of a linear program (LP) whose right-hand side is parameterized by the first-stage decisions.
Such a structure leads to a convex, piecewise-linear value function. This, in turn, enables the exploitation of well-established algorithms, such as L-shaped and  (augmented) Lagrangian methods. 
Relaxing these assumptions destroys convexity and piecewise linearity, which are crucial for ensuring global optimality. 

In recent years, the focus of research in two-stage SPs has begun to shift beyond the traditional convex frameworks toward models with nonconvex recourse functions. 
A notable structure is when the objective function of the second-stage problem is linearly parametrized by the first-stage decision variables \citep{bomze2022two}. 
The implicit convex-concave property of the recourse function was exploited in \cite{liu2020two} to develop an iterative algorithm that combines regularization and convexification to obtain iterates by solving convex subproblems. Lifting the scenario problem to the space of $(\bs{x},\bs{y})$ and using its partial Moreau envelope, \citet{li2024decomposition} developed a decomposition framework that generates strongly convex quadratic approximations of the recourse function. Moreover, an outer loop successively updates the parameter of the partial Moreau envelope.
In pursuit of global optimality, \citet{zhong2024towards} constructed polynomial lower approximations to the nonconvex recourse via linear conic optimization over nonnegative polynomial cones, which further employed semidefinite relaxations given the complex structure of this cone. \cite{kang2025bi} developed a cutting-plane algorithm by relying on the Lagrangian dual, which is finitely convergent when all first-stage variables are binary.

\subsubsection{Distinguishing Features of this Work}
Similar to existing DRO models with endogenous uncertainty, we capture decision dependency with a decision-dependent ambiguity set. Nevertheless, instead of focusing on a particular type of ambiguity set, e.g., a discrepancy- or moment-based model, we consider a polyhedral ambiguity set. 
Moreover, unlike most studies that assume a convex recourse function, we also study a nonconvex one. 
The proposed finitely-convergent cutting-plane algorithm for the problem with the nonconvex recourse is applicable to solve a general two-stage SP  when the cost vector---or equivalently, the recourse matrix---in the second-stage problem is simultaneously (linearly) bi-parameterized by the first-stage decision variables and uncertain parameters.   
However, unlike the above studies, outer approximations of the recourse function are linear in the space of first-stage decision variables, and valid inequalities are obtained by solving LPs.  Furthermore, unlike \cite{kang2025bi}, the proposed algorithm in this paper is finitely convergent even when the first-stage variables are mixed-integer. 


\subsection{Organization}

The rest of this paper is outlined as follows. 
In Section \ref{sec: formulation}, we introduce the studied class of DRO problems with a decision-dependent ambiguity set. 
In Section \ref{sec: reformulation}, we reformulate the studied problem as a nonconvex two-stage stochastic MINLP. 
In Section \ref{sec: solution}, we propose a decomposition-based cutting-plane algorithm to solve the resulting DRO problems optimally (or near-optimally) and establish its finite convergence. In Section \ref{sec: numerical}, we present numerical experiments to test the efficacy of the solution algorithm and the quality of resulting solutions. We end with conclusions and a discussion of future work in Section \ref{sec: discussion}. All the proofs and technical lemmas are relegated to the Electronic Companion (EC).

\noindent {\bf Notation}.  
Throughout this paper, vectors are denoted by boldface lowercase letters and matrices are denoted by boldface uppercase letters. Sets are denoted by calligraphic uppercase letters. 
For a set $\Cs{B} \subseteq \Bs{R}^{d}$,  $\conv{\Cs{B}}$ denote the convex hull of $\Cs{B}$. 
Let $\bs{e}_{i}$ be the $i$-th unit vector and $\bs{e}$ be a vector of ones 
 in $\Bs{R}^{d}$. 
For any $n \in \Bs{N}$, we refer to the ordered index set $\{1,\dotsc,n\}$ by $[n]$.

\section{Problem Formulation}
\label{sec: formulation}

The following decision-dependent DRO problem describes the central problem under consideration 
\begin{equation}
\label{eq: DRO_Obj}
\min_{\bs{x} \in \Cs{X} } \ \bs{c}^{\top} \bs{x}+ \bs{x}^{\top} \bs{C} \bs{x} + \max_{P \in \Cs{P}(\bs{x})} \ 	\ep{P}{h(\bs{x},\bs{\xi})},
\end{equation} 
where $\bs{x}$ is the decision vector in a nonempty and compact, deterministic mixed-integer feasible region  $\Cs{X}:= \Set*{\bs{x} \in  \Bs{Z}^{n_{1}} \times \Bs{R}^{n_x-n_{1}} }{ \bs{A}\bs{x} \ge \bs{d}, \bs{x} \ge \bs{0}}$, with $n_1 \ge 0$ integer variables.  We define a random vector $\bs{\xi} \in \Xi \subseteq \Bs{R}^{n_{\xi}}$.  
Moreover,  $h(\bs{x},\bs{\xi}): \Cs{X} \times \Xi \mapsto \Bs{R}$ is a random cost function. 
For a given $\bs{x} \in \Cs{X}$, $\Cs{P}(\bs{x})$ is a decision-dependent set of probability distributions, with $P$ as an element of this set. 
We further assume 
\begin{equation}
    \label{eq: recourse}
    h(\bs{x}, \bs{\xi})= \min_{\bs{y} \in \Cs{Y}(\bs{x}, \bs{\xi})} \psi_{0}(\bs{x}, \bs{y}, \bs{\xi}), 
\end{equation}
where  
     $\Cs{Y}(\bs{x},\bs{\xi}):=\Set*{\bs{y} \in  \Bs{R}^{n_y}}{\bs{D}\bs{y} \ge \bs{B}\bs{x}+\bs{b}, \; \bs{y} \ge \bs{0}}$, 
with matrices $\bs{D}$ and $\bs{B}$, and vector $\bs{b}$ of appropriate dimensions. 
Throughout the paper, we make the following assumptions. 

  {\bf A1 (Finite Sample Space)} \label{assum: finitespace} Each distribution  $P \in \Cs{P}(\bs{x})$ has  a decision-independent finite support $\Xi=\{\bs{\xi}_{\omega}\}_{\omega=1}^{N}$, for all $\bs{x} \in \Cs{X}$, where $N$ is the number of scenarios.
  
 {\bf A2 (Nonempty Ambiguity Set)} \label{assum: nonemptyP} For a fixed $\bs{x} \in \Cs{X}$, $\Cs{P}(\bs{x})$ is polyhedral and nonempty.

Given Assumption {\bf (A1)}, any $P \in \Cs{P}(\bs{x})$ can be presented by 
$\bs{p}=[p_{1}, \ldots, p_{N}]^{T} \in \Bs{R}^{N}$. Assumption {\bf (A2)} implies that the optimal value to the inner maximization $\max_{P \in \Cs{P}(\bs{x})} \ 	\ep{P}{h(\bs{x},\bs{\xi})}$ in \eqref{eq: DRO_Obj} is a proper function in $\bs{x} \in \Cs{X}$; hence, problem \eqref{eq: DRO_Obj} has an optimal solution with a bounded optimal value. Nonetheless, the nonemptiness of the ambiguity set is not a restrictive assumption, as one can project onto the plane of $\bs{x}$-variables to recover $\Set*{\bs{x} \in \Cs{X}}{\exists P \in \Cs{P}(x)}$. This can be achieved by cutting off points that result in infeasibility, using extreme rays of the dual feasible space. 

\subsection{Decision-Dependent Ambiguity Set of Probability Distributions}

We consider a general decision-dependent ambiguity set with a polyhedral structure, namely with a {\it generalized moment and measure inequalities}  \citep{rahimian2022frameworks} as: 
\begin{equation}
    \label{eq: simple-moment}
    \Cs{P}(\bs{x}):=
    \left\lbrace \bs{p} \ge \bs{0} \middle\vert
    \begin{array}{l}
        \ul{p}_{\omega} \le  p_{\omega} \le \ol{p}_{\omega}, \; \omega \in [N], \\
        \sum_{\omega \in [N]} p_{\omega} g^{i}(\bs{\xi}_{\omega}) \le  \vartheta^{i}(\bs{x}), \; i \in [s]
    \end{array}
    \right\rbrace,
\end{equation}
for $\bs{x} \in \Cs{X}$, where $g^{i}: \Xi \mapsto \Bs{R}$, $i \in [s]$, with $s \ge 1$. The decision-independent version of \eqref{eq: simple-moment} has been extensively studied in the literature, see, e.g., \cite{shapiro2004minmax, popescu2005semidefinite, perakis2008regret,mehrotra2014semi}, and it subsumes some classes of discrepancy- and moment-based ambiguity sets under a finite sample space, such as total variation- and Wasserstein-based sets \citep{bansal2018decomposition}.
As a shorthand notation, we use $\bs{g}=[g^{1}, \ldots, g^{s}]^{\top}$ and $\bs{\vartheta}(\bs{x})=[\vartheta^{1}(\bs{x}), \ldots, \vartheta^{s}(\bs{x})]^{\top}$. 
The first set of constraints in \eqref{eq: simple-moment} enforce a preference relationship between probability distributions, i.e., lower and upper bounds on the probabilities. 
The second set of constraints in \eqref{eq: simple-moment} bound some functions $\bs{g}$ of the random vector $\bs{\xi}$, for example, $\bs{g}$ can represent moments of $\bs{\xi}$.  To ensure that $\bs{p}$ is a probability distribution, we set $\vartheta^1(\bs{x})=1$, $g^1(\cdot)=1$, $\vartheta^2(\bs{x})=-1$, and $g^2(\cdot)=-1$, in the above definition of $\Cs{P}(\bs{x})$, for all $\bs{x} \in \Cs{X}$. 
For instance, a moment-based ambiguity set may be represented as:
\begin{equation}
\label{eq: moment}
    \Cs{P}_{\textrm{M}}(\bs{x}) = 
    \left\lbrace \bs{p} \ge \bs{0} \middle\vert
    \begin{array}{l}
         (1- \tau_\mu) \bs{\mu}_{0}(\bs{x}) \le  \sum\limits_{\omega \in [N]} p_{\omega} \bs{\xi}_{\omega}    \le (1+\tau_\mu)\bs{\mu}_{0}(\bs{x}), \\  
        \underline{\tau}_\sigma \big(\bs{\sigma}_{0}(\bs{x})^2 + \bs{\mu}_{0}(\bs{x})^2\big) \le \sum\limits_{\omega \in [N]} p_{\omega} \bs{\xi}_{\omega}^{2}\le \overline{\tau}_\sigma \big(\bs{\sigma}_{0}(\bs{x})^2 + \bs{\mu}_{0}(\bs{x})^2\big), \\
        \sum\limits_{\omega \in [N]} p_{\omega} = 1
    \end{array}
    \right\rbrace,
\end{equation}
where $\bs{\mu}_{0}(\bs{x})$ and $\bs{\sigma}_{0}(\bs{x})$ denote the vector of (nominal) decision-dependent mean and standard deviation of the random vector $\bs{\xi}$. Moreover, parameters $\tau_\mu$, $\underline{\tau}_\sigma$, and $\overline{\tau}_\sigma$ control the conservatism of resulting robust decisions by adjusting the maximum allowable deviations on the nominal first- and second-order moments, with $0\le \underline{\tau}_\sigma \le 1 \le \overline{\tau}_\sigma$.

\subsection{Recourse Function}
In this paper, we consider the objective function  
\begin{equation}
\label{eq: nonconvex-recourse}
    \psi_{0}(\bs{x}, \bs{y}, \bs{\xi})= \bs{q}^{\top}\bs{y} + \bs{x}^{\top}\bs{L} \bs{y}.
\end{equation}
The setting with $\bs{L}=\bs{0}$ is predominantly studied in the literature due to the desirable convexity and piecewise linearity of the resulting recourse function. For instance, the facility location problem introduced in Section \ref{sec: FL_app} has such an objective function. 
In this paper, we are additionally interested in the setting that $\bs{L} \neq \bs{0}$ due to its emergence in problems with endogenous uncertainty. For instance, the newsvendor problem introduced in Section \ref{sec: NV_app} has such an objective function. 
This setting is fundamentally different from the classical paradigm for two-stage SPs with continuous recourse in that the first-stage decision variables $\bs{x}$ appear not only in the constraints of the second-stage problem, but also in the objective function. As a result, the recourse function $h(\bs{x}, \bs{\xi})$ is nonconvex since $\psi_{0}(\bs{x}, \bs{y}, \bs{\xi})$ is not jointly convex in $(\bs{x},\bs{y})$, which introduces significant computational challenges (see Section \ref{sec: solution}). 

In problem \eqref{eq: recourse}, $\bs{\xi}$ includes all parameters that describe the recourse function $h(\bs{x},\bs{\xi})$, i.e.,  vectors $\bs{q}$, $\bs{b}$ and matrices $\bs{B}$, $\bs{D}$, and  $\bs{L}$. Unless otherwise stated explicitly, we assume $\bs{L} \neq \bs{0}$ and we refer to the structure in \eqref{eq: nonconvex-recourse} as a simultaneously {\it (linearly) bi-parameterized} objective in both first-stage decision variables and uncertain parameters.  Furthermore, we make the following assumption throughout the paper.

{\bf A3 (Complete Recourse)} \label{assum: finite-h} For $\bs{x} \in \Bs{R}^{n_x}$ and $\omega \in [N]$, we have a complete recourse for problem \eqref{eq: recourse}  and $\Cs{Y}(\bs{x},\bs{\xi}_{\omega})$ is a bounded set. 

\begin{remark}

\label{rem: bounded_dual}     
Assumption {\bf (A3)} implies relatively complete recourse; i.e., $\Cs{Y}(\bs{x},\bs{\xi}_{\omega})$ is feasible for any $\bs{x} \in \Cs{X}$ and $\omega \in [N]$. Thus, $ h(\bs{x},\bs{\xi}_\omega)$ is bounded from above and $\Set*{\bs{\pi} \ge \bs{0}}{\bs{\pi}^{\top} \bs{D}_{\omega} \le \bs{0}^{\top}, \ \bs{\pi}^{\top} (\bs{B}_{\omega} \bs{x} + \bs{b}_{\omega}) \ge \bs{0}} =\{\bs{0}\}$ by Farkas' Lemma. Hence, an optimal solution to the corresponding dual problem of \eqref{eq: recourse} is attained at an extreme point, showing that $ h(\bs{x},\bs{\xi}_{\omega})$ is bounded from below (by weak duality). In addition, given that we have a complete recourse for every $\bs{x} \in \Bs{R}^{n}$, then $\Set*{\bs{\pi} \ge \bs{0}}{\bs{\pi}^{\top} \bs{D}_{\omega} \le \bs{0}^{\top}}=\{\bs{0}\}$; so, the set of optimal dual solutions is bounded. 
\hfill	$\blacksquare$
\end{remark}

\section{Two-Stage Stochastic MINLP Reformulation}
\label{sec: reformulation}

A crucial step in developing efficient solution algorithms for problem \eqref{eq: DRO_Obj} is to reformulate it into an equivalent form that is more amenable to computation.
In this section, we provide a reformulation of problem \eqref{eq: DRO_Obj} with a decision-dependent polyhedral ambiguity set in the form of \eqref{eq: simple-moment}.



\begin{theorem}
    \label{thm: simple-moment-inner-lag-dual}
    Problem \eqref{eq: DRO_Obj} can be reformulated  as:
    \begin{equation}
        \label{eq: reform-simple-moment}
        \min_{\bs{x},\bs{\lambda}} \Big\{  \bs{c}^{\top} \bs{x}+ \bs{x}^{\top} \bs{C} \bs{x}+  \bs{\lambda}^{\top}\bs{\vartheta}(\bs{x}) 
        	+\sum_{\omega \in [N]} G_{\omega}(\bs{x}, \bs{\lambda}) : \bs{x}\in \Cs{X}, \; \bs{\lambda} \ge \bs{0} \Big\},
        \end{equation}
        where 
        \begin{equation}
            \label{eq: Q_k}
            G_{\omega}(\bs{x},\bs{\lambda}):=\varphi_{\omega}\Big [ h(\bs{x},\bs{\xi}_{\omega}) - \bs{\lambda}^{\top} \bs{g}(\bs{\xi}_{\omega}) \Big],
        \end{equation}
        with $\varphi_{\omega}[z]=    \ol{p}_{\omega} (z)_{+} -   \ul{p}_{\omega} (-z)_{+}$. 
\end{theorem}

Note that $G_{\omega}(\bs{x},\bs{\lambda})$ is well-defined by the finiteness of $h(\bs{x},\bs{\xi}_{\omega})$ (see Remark \ref{rem: bounded_dual}). More importantly, $G_{\omega}(\bs{x},\bs{\lambda})$ is nonconvex in $(\bs{x},\bs{\lambda})$. 
Nonetheless, as $\varphi_{\omega}[z]$ is monotonically nondecreasing (see Lemma \ref{lem: convex_function}), a a two-stage reformulation of \eqref{eq: DRO_Obj} can be formed.

\begin{corollary}
    \label{cor: reform-recourse}
    Problem \eqref{eq: DRO_Obj}  can be reformulated as the following two-stage stochastic MINLP:
    \begin{equation}
        \label{eq: reform-recourse}
        \min_{\bs{x},\bs{\lambda}} \Big\{ \bs{c}^{\top} \bs{x}+ \bs{x}^{\top} \bs{C} \bs{x}+   \bs{\lambda}^{\top}\bs{\vartheta}(\bs{x}) 
        	+\sum_{\omega \in [N]} G_{\omega}(\bs{x}, \bs{\lambda}) : \bs{x}\in \Cs{X}, \; \bs{\lambda} \ge \bs{0} \Big\} 
        \end{equation}
        where 
        \begin{equation}
            \label{eq: Q_k_recourse}
                G_{\omega}(\bs{x},\bs{\lambda}) =\min_{\bs{y}, \mu, \gamma } \Big\{\gamma \ol{p}_{\omega} - \mu \ul{p}_{\omega} : 
                \gamma - \mu \ge \psi_{0}(\bs{x},\bs{y}, \bs{\xi}_{\omega}) - \bs{\lambda}^{\top} \bs{g}(\bs{\xi}_{\omega}), \; 
                \bs{y} \in \Cs{Y}(\bs{x},\bs{\xi}_{\omega}), \;
                \gamma, \mu \ge 0\Big\}. 
        \end{equation}
\end{corollary}


It can be seen from Corollary \ref{cor: reform-recourse} that in addition to already existing nonlinear terms $\bs{x}^\top \bs{C} \bs{x}$ and $\bs{x}^\top \bs{L} \bs{y}$, dualization introduces a nonconvex, nonlinear coupling $\bs{\lambda}^\top \bs{\vartheta}(\bs{x})$, 
which poses a computational challenge. 
Moreover, one can interpret \eqref{eq: reform-recourse} as a two-stage SP, where $(\bs{x}, \bs{\lambda})$ are the first-stage decisions and $(\bs{y}_{\omega}, \mu_{\omega}, \gamma_{\omega})$, $\omega \in [N]$, are the second-stage decisions. 
In Section \ref{sec: solution}, we leverage this structure to handle the abovementioned computational challenges. 



Next, we present an extensive deterministic equivalent formulation (DEF) of \eqref{eq: reform-recourse}. 

\begin{corollary}
    \label{cor: reform-simple-moment-recourse}
    Problem \eqref{eq: DRO_Obj} can be reformulated  as the following nonconvex MINLP:
    \begin{equation}
        \label{eq: reform-simple-moment-recourse-linearized}
        \begin{aligned}
        \min_{\bs{x},\bs{\lambda},\bs{y}_1, \ldots, \bs{y}_{N}, \bs{\gamma}, \bs{\mu}} \ &  \bs{c}^{\top} \bs{x}+ \bs{x}^{\top} \bs{C} \bs{x} + \bs{\lambda}^{\top}\bs{\vartheta}(\bs{x}) 
        		+\sum_{\omega \in [N]} (\gamma_{\omega} \ol{p}_{\omega} - \mu_{\omega} \ul{p}_{\omega}) \\
        \st \quad &  \bs{x}\in \Cs{X}, \; \bs{\lambda} \ge \bs{0}, \\ 
        & \gamma_{\omega} - \mu_{\omega} \ge  \psi_{0}(\bs{x},\bs{y}_{\omega}, \bs{\xi}_{\omega}) - \bs{\lambda}^{\top} \bs{g}(\bs{\xi}_{\omega}), \quad  \omega \in [N], \\
        & \bs{y}_{\omega} \in \Cs{Y}(\bs{x},\bs{\xi}_{\omega}), \quad  \omega \in [N], \\
        & \bs{\gamma}, \bs{\mu}\ge \bs{0}.
        \end{aligned}
    \end{equation}
\end{corollary}


\section{An Exact Cutting-Plane Algorithm}
\label{sec: solution}



The two-stage structure outlined in Corollary \ref{cor: reform-recourse} renders problem \eqref{eq: reform-recourse} suitable for decomposition-based solution approaches. 
In this section, we develop a cutting-plane algorithm to solve \eqref{eq: reform-recourse} and establish its finite convergence (Theorem \ref{thm: DD-finite-convergence-nonconvex}). 
To handle the nonconvex $G(\bs{x},\bs{\lambda})$, this algorithm relies on successively generating disjunctive cuts in the space of first-stage decision variables. Each set in the underlying disjunction represents an epigraphic reformulation of the (recourse) Lagrangian function given a vector of dual values, formed at a boundary of recourse decisions which are {\it bilinearly} connected to first-stage decision variables. 

The proposed iterative algorithm requires solving a nonconvex master problem with multilinear terms. 
Various techniques have been developed for globally optimizing such problems \citep{fampa2021convexification,burer2012non}, some of which are implemented in branch-and-bound schemes within commercial nonconvex solvers. These solvers typically rely on bounds on variables. In Lemma~\ref{lem: boundedness_dual}, we show that there exists a compact set $\Cs{L}\subseteq \Bs{R}_{+}^{s}$ such that restricting $\bs{\lambda}$ to $\Cs{L}$ preserves the optimal value of \eqref{eq: reform-simple-moment}. 
In the numerical experiments in Section~\ref{sec: numerical}, we use commercial solvers supplied with variable bounds to solve the master problems. For our theoretical analysis, we assume access to an oracle that returns an $\epsilon$-optimal solution. A related oracle implementation for nonconvex bilinear programs is provided by \citet{rahimian2020finitely}, who developed a finitely-convergent cutting-plane algorithm for computing $\epsilon$-optimal solutions.

{\bf A4 ($\epsilon$-Optimal Global Solver)} For any $\epsilon>0$, there is an oracle that solves a nonconvex problem  $z^*=\min_{(\bs{x},\bs{\lambda}, \theta) \in \Cs{S}} \  \bs{c}^{\top} \bs{x}+ \bs{x}^{\top} \bs{C} \bs{x} + \bs{\lambda}^{\top}\bs{\vartheta}(\bs{x})
		+ \theta $ to $\epsilon$-optimality over a mixed-integer set $\Cs{S}$. That is, it obtains a solution  $(\hat{\bs{x}},\hat{\bs{\lambda}},\hat{\theta})$ such that $\bs{c}^{\top} \hat{\bs{x}}+ \hat{\bs{x}}^{\top} \bs{C} \hat{\bs{x}} + \bs{\vartheta}(\hat{\bs{x}})^{\top} \hat{\bs{\lambda}} + \hat{\theta} \le z^{*} + \epsilon$.

\subsection{Lagrangian Dual Formulation}
\label{sec: alg_nonconvex}
The nonconvex $G_{\omega}(\bs{x},\bs{\lambda})$, defined in \eqref{eq: Q_k_recourse}, is the value function of an LP as: 
\begin{subequations}
    \label{eq: Q_k_nonconvex-recourse}
    \begin{align}
        G_{\omega}(\bs{x},\bs{\lambda}) =\min_{\bs{y}, \mu, \gamma} & \ \gamma\ol{p}_{\omega} - \mu \ul{p}_{\omega} \\
        & \gamma - \mu - \bs{q}_{\omega}^{\top}\bs{y} -\bs{x}^{\top} \bs{L}_{\omega} \bs{y} \ge - \bs{\lambda}^{\top} \bs{g}(\bs{\xi}_{\omega}), \label{nonconvexsigma}\\
        & \bs{D}_{\omega}\bs{y} \ge \bs{B}_{\omega}\bs{x}+\bs{b}_{\omega}, \label{nonconvexpi}\\ 
        & \bs{y} \ge \bs{0}, \; \gamma, \mu \ge 0. 
    \end{align}
\end{subequations}
Let $\sigma$ and $\bs{\pi}$ be the Lagrange multipliers corresponding to \eqref{nonconvexsigma} and \eqref{nonconvexpi}, respectively. Moreover, $\bs{\pi}_{\bs{y}}$, $\pi_{\gamma}$, and $\pi_{\mu}$ are the Lagrange multipliers corresponding to $\bs{y} \ge \bs{0}$, $\gamma \ge 0$, and $\mu \ge 0$. 
The Lagrangian function can be written as: 
\begin{align*}
    \Fs{L}_{\omega}(\bs{x},\bs{\lambda},\bs{y}, \mu, \gamma, \sigma, \bs{\pi}, \bs{\pi}_{\bs{y}}, \pi_{\gamma}, \pi_{\mu}) 
    & :=  
    -\sigma  \bs{\lambda}^{\top} \bs{g}(\bs{\xi}_{\omega}) + \bs{\pi}^{\top} (\bs{B}_{\omega}\bs{x}+\bs{b}_{\omega}) +  \gamma ( \ol{p}_{\omega} -  \pi_{\gamma} -\sigma )    \\
    & \qquad - \mu (\ul{p}_{\omega} +\pi_{\mu} - \sigma)+ (\sigma \bs{q}_{\omega}^{\top} - \bs{\pi}_{\bs{y}}^{\top} - \bs{\pi}^{\top} \bs{D}_{\omega} + \sigma  \bs{x}^{\top} \bs{L}_{\omega})\bs{y}. 
\end{align*}
Consider the Lagrangian dual function  
\begin{equation*}
    Q_{\omega}(\bs{x},\bs{\lambda}, \sigma, \bs{\pi}, \bs{\pi}_{\bs{y}}, \pi_{\gamma}, \pi_{\mu}):=\min_{\bs{y}, \mu, \gamma} \ \Fs{L}_{\omega}(\bs{x},\bs{\lambda},\bs{y}, \mu, \gamma, \sigma, \bs{\pi}, \bs{\pi}_{\bs{y}}, \pi_{\gamma}, \pi_{\mu}).  
\end{equation*}
By strong duality and using the compact set $\Cs{L}$ (Lemma~\ref{lem: boundedness_dual}), \eqref{eq: reform-recourse} can be reformulated as:
\begin{equation}
\label{eq: DD-nonconvex-recourse-reform}
\begin{aligned}
\min_{\bs{x},\bs{\lambda}, \bs{\theta}} \ &   \bs{\lambda}^{\top}\bs{\vartheta}(\bs{x})
		+ \sum_{\omega \in [N]} \theta_{\omega} \\
        \st \quad  &  \bs{x}\in \Cs{X}, \; \bs{\lambda} \in \Cs{L},\\
        & \theta_{\omega} \ge Q_{\omega}(\bs{x},\bs{\lambda}, \sigma, \bs{\pi}, \bs{\pi}_{\bs{y}}, \pi_{\gamma}, \pi_{\mu}), \quad \omega \in [N], \; \forall \bs{\pi}, \bs{\pi}_{\bs{y}} \ge \bs{0}, \; \sigma, \pi_{\gamma}, \pi_{\mu} \ge 0. 
\end{aligned}
\end{equation}
Let us consider the Lagrangian dual problem 
\begin{align}
G_{\omega}(\bs{x},\bs{\lambda})
&= \max_{\bs{\pi},\bs{\pi}_{\bs{y}} \ge \bs{0},
          \sigma,\pi_{\gamma},\pi_{\mu} \ge 0}
   Q_{\omega}(\bs{x},\bs{\lambda},\sigma,\bs{\pi},
              \bs{\pi}_{\bs{y}},\pi_{\gamma},\pi_{\mu})
   \notag\\
&= \max_{\bs{\pi},\bs{\pi}_{\bs{y}} \ge \bs{0},
          \sigma,\pi_{\gamma},\pi_{\mu} \ge 0}
   \Bigl\{
   -\sigma \bs{\lambda}^{\top}\bs{g}(\bs{\xi}_{\omega})
   + \bs{\pi}^{\top}(\bs{B}_{\omega}\bs{x}+\bs{b}_{\omega})
   \notag\\
&\qquad\qquad {}
   + \min_{\bs{y},\mu,\gamma}
   \gamma(\ol{p}_{\omega}-\pi_{\gamma}-\sigma)
   - \mu(\ul{p}_{\omega}+\pi_{\mu}-\sigma)
   + \bigl(
       \sigma\bs{q}_{\omega}^{\top}
       - \bs{\pi}_{\bs{y}}^{\top}
       - \bs{\pi}^{\top}\bs{D}_{\omega}
       + \sigma\bs{x}^{\top}\bs{L}_{\omega}
     \bigr)\bs{y}
   \Bigr\}
   \notag\\
\begin{split}
&= \begin{aligned}[t]
   \max_{\bs{\pi} \ge \bs{0},\sigma \ge 0}
   \quad &
   -\sigma\bs{\lambda}^{\top}\bs{g}(\bs{\xi}_{\omega})
   +\bs{\pi}^{\top}(\bs{B}_{\omega}\bs{x}+\bs{b}_{\omega})\\
   \st \quad &
   \ul{p}_{\omega}\le \sigma\le \ol{p}_{\omega},\\
   & -\sigma\bs{q}_{\omega}^{\top}
     +\bs{\pi}^{\top}\bs{D}_{\omega}
     -\sigma\bs{x}^{\top}\bs{L}_{\omega}
     \le \bs{0}^{\top}.
   \end{aligned}
\end{split}
\label{eq: Lagrangian}
\end{align}
where $\pi_{\mu}$ and $\pi_{\gamma}$, are the slack variables to constraints $\ul{p}_{\omega} \le \sigma$ and $- \ol{p}_{\omega} \le -\sigma$, respectively, and $\bs{\pi}_{\bs{y}}$ is the slack variable to the last constraint of \eqref{eq: Lagrangian}. 
\begin{proposition}
\label{prop: Lagrangian}
Given $\hat{\bs{x}} \in \Cs{X}$ and $\hat{\bs{\lambda}} \in \Cs{L}$, let $(\hat{\sigma},\hat{\bs{\pi}})$ 
be optimal Lagrange multipliers to \eqref{eq: Lagrangian}
with zero duality gap. 
Then, the Lagrangian dual function can be written as:
\begin{equation}
Q_{\omega}(\bs{x},\bs{\lambda}, \hat{\sigma}, \hat{\bs{\pi}})  
= - \hat{\sigma}   
 \bs{\lambda}^{\top} \bs{g}(\bs{\xi}_{\omega}) +  \hat{\bs{\pi}}^{\top} (\bs{B}_{\omega}\bs{x}+\bs{b}_{\omega}) + \min_{\bs{y}_{\omega}} \  \hat{\sigma}  (\bs{x} - \hat{\bs{x}})^{\top} \bs{L}_{\omega}\bs{y}_{\omega}. \label{eq: Lagrangian-nonconvex}
\end{equation}
\end{proposition}
Consequently, an optimal solution of the Lagrangian dual function \eqref{eq: Lagrangian-nonconvex} only depends on variables $\bs{y}$ that are bilinearly connected to $\bs{x}$ in the description of the objective function $\psi_{0}(\bs{x}, \bs{y}, \bs{\xi}_{\omega})=\bs{q}_{\omega}^{\top}\bs{y}+\bs{x}^{\top}\bs{L}_{\omega}\bs{y}$. 

\begin{lemma}
    \label{lem: optimality_conditions}
    Suppose that Assumption {\bf (A3)} holds. For $\omega \in [N]$, let $\underline{\bs{y}}_{\omega} \le \bs {y}_{\omega} \le \overline{\bs{y}}_{\omega}$, and $\bs{l}_{\omega,j}$ be the $j$-th column of matrix $\bs{L}_{\omega}$. Let $\bs{y}^*_{\omega}$ be an optimal solution to the  Lagrangian dual function \eqref{eq: Lagrangian-nonconvex}. If $(\bs{x} - \hat{\bs{x}})^{\top} \bs{l}_{\omega,j} \ge 0$, then $y^*_{\omega,j}=\underline{y}_{\omega,j}$. Otherwise, if $(\bs{x} - \hat{\bs{x}})^{\top} \bs{l}_{\omega,j} \le 0$, then $y^*_{\omega,j}=\overline{y}_{\omega,j}$. 
\end{lemma}

\subsection{A Disjunctive Cutting-Plane Method}

A general framework of the disjunctive cutting-plane scheme to solve \eqref{eq: DRO_Obj} is given in Algorithm \ref{alg: DD-nonconvex-recourse}. 
one can form a restricted master problem for \eqref{eq: DRO_Obj} after $t$ iterations as follows: 
\begin{equation}
\label{eq: DD-nonconvex-recourse-master}
z^{t}=\min_{(\bs{x},\bs{\lambda}, \bs{\theta}) \in \Cs{S}^{t}} \  \bs{c}^{\top} \bs{x}+ \bs{x}^{\top} \bs{C} \bs{x} + \bs{\lambda}^{\top}\bs{\vartheta}(\bs{x})
		+ \sum_{\omega \in [N]} \theta_{\omega},
\end{equation}
where 
\begin{equation*}
    \Cs{S}^{t}= \left\lbrace (\bs{x},\bs{\lambda}, \bs{\theta}) \middle\vert
    \begin{array}{l}
        \bs{x}  \in \Cs{X}, \; \bs{\lambda} \in \Cs{L}, \\
        \bs{x}^{\top} \bs{\alpha}_{\bs{x},\omega}^{k} + \bs{\lambda}^{\top} \bs{\alpha}_{\bs{\lambda},\omega}^{k}  + \theta_{\omega} \alpha_{\theta,\omega}^{k}  \ge \rho_{\omega}^{k}, \; \omega \in [N], \; k \in \{0\} \cup [t-1]
    \end{array}
    \right\rbrace.
\end{equation*}
Here, $(\bs{x}^{t}, \bs{\lambda}^{t}, \bs{\theta}^{t})$ is a solution obtained from solving \eqref{eq: DD-nonconvex-recourse-master} at iteration $t$, and  \linebreak $(\bs{\alpha}_{\bs{x},\omega}^{k}, \bs{\alpha}_{\bs{\lambda},\omega}^{k}, \alpha_{\theta,\omega}^{k}, \rho_{\omega}^{k})$ are the corresponding cut coefficients, derived via Algorithm \ref{alg: SepCuts}, $\omega \in [N]$ and $k \in \{0\} \cup [t-1]$. 
Note that $\Cs{S}^{0}$ is the initial set, with no optimality cuts. 

\begin{algorithm}[!tb] 
    \SetAlgoLined
        \KwIn{An initial solution   $(\bs{x}^{0},\bs{\lambda}^{0}, \bs{\theta}^{0}) \in \Cs{X}\times\Cs{L}\times\Bs{R}^{N}$ and $\epsilon>0$ for the optimality tolerance.}
	    \KwOut{An $\epsilon$-optimal solution and the $\epsilon$-optimal value.}
	
		Initialization: Set $t  \gets 1$, $\Omega \gets \emptyset$, $\mathrm{LB} \gets - \infty$, $\mathrm{UB} \gets + \infty$. Add initial cuts (e.g., $\theta_{\omega} \ge 0$), if available, to $\Cs{S}^{0}$.

		\While{$\mathrm{UB} - \frac{\epsilon}{2}>  \mathrm{LB} $ }{ 
		    \For{\textbf{each} $\omega \in [N]$} 
		    {   
		        Obtain $G_{\omega}(\bs{x}^{t-1}, \bs{\lambda}^{t-1})$ by solving \eqref{eq: Q_k_nonconvex-recourse} and $(\bs{\pi}_{\omega}^{t-1}, \sigma_{\omega}^{t-1})$ by solving \eqref{eq: Lagrangian}  at $(\bs{x}^{t-1}, \bs{\lambda}^{t-1})$. \label{line: duals}
          
                Call the procedure \texttt{SepCuts}$(\bs{x}^{t-1}, \bs{\lambda}^{t-1}, \theta^{t-1}_{\omega}; \bs{\pi}_{\omega}^{t-1}, \sigma_{\omega}^{t-1})$ to obtain $(\texttt{viol}, \bs{\alpha}_{\bs{x},\omega}, \bs{\alpha}_{\bs{\lambda},\omega}, \alpha_{\theta,\omega}, \rho_{\omega})$. \label{line: cut}
		        
		        \If{\texttt{viol}$=$TRUE}{
                        $\Omega \gets \Omega \cup \{\omega\}$.
                    }
		}
		   
		        $\mathrm{UB} \gets \min\{\mathrm{UB}, \bs{c}^{\top} \bs{x}^{t-1}+ (\bs{x}^{t-1})^{\top} \bs{C} \bs{x}^{t-1}  + \bs{\vartheta}(\bs{x}^{t-1})^{\top}\bs{\lambda}^{t -1}  + \sum_{\omega \in [N]} G_{\omega}(\bs{x}^{t-1}, \bs{\lambda}^{t-1})\}$.

		    Let  $\Cs{S}^{t} \gets \Cs{S}^{t-1} \cap \Set*{(\bs{x},\bs{\lambda}, \bs{\theta})}{ \bs{\alpha}_{\bs{x},\omega}^{\top} \bs{x} +\bs{\alpha}_{\bs{\lambda},\omega}^{\top} \bs{\lambda} + \alpha_{\theta,\omega} \theta_{\omega} \ge \rho_{\omega},  \; \omega \in \Omega}$.

		    Solve restricted master problem \eqref{eq: DD-nonconvex-recourse-master} using $\Cs{S}^{t}$ and obtain an $\frac{\epsilon}{2}$-optimal solution $(\bs{x}^{t}, \bs{\lambda}^{t}, \bs{\theta}^{t})$. \label{line: iterates}
		    
		    Let  $\mathrm{LB} \gets \bs{c}^{\top} \bs{x}^t+ (\bs{x}^t)^{\top} \bs{C} \bs{x}^t + \bs{\vartheta}(\bs{x}^t)^{\top} \bs{\lambda}^{t} 
		+ \sum_{\omega \in [N]} \theta_{\omega}^{t}$.
		    
		    Set $t\gets t+1$, $\Omega \gets \emptyset$.
		    
		}
		\Return $(\bs{x}^{t-1}, \bs{\lambda}^{t-1})$ and $\mathrm{UB}$.
		
		\caption{Disjunctive cutting-plane algorithm for problem \eqref{eq: DRO_Obj}.}
  \label{alg: DD-nonconvex-recourse}
\end{algorithm}

We now explain how to obtain a disjunctive cut given a solution $\hat{\bs{x}} \in \Cs{X}$ and $\hat{\bs{\lambda}} \in \Cs{L}$. Let $\hat{\bs{y}}_{\omega}$ be a combination of lower and upper bounds on variables $\bs{y}$ that attains Lagrangian dual function \eqref{eq: Lagrangian-nonconvex} (Lemma \ref{lem: optimality_conditions}). 
Thus, 
    $\theta_{\omega} \ge -\hat{\sigma}   
 \bs{\lambda}^{\top} \bs{g}(\bs{\xi}_{\omega}) +  \hat{\bs{\pi}}^{\top} (\bs{B}_{\omega}\bs{x}+\bs{b}_{\omega}) + \hat{\sigma}  (\bs{x} - \hat{\bs{x}})^{\top} \bs{L}_{\omega}\hat{\bs{y}}_{\omega}$
is a valid linear inequality for \eqref{eq: DD-nonconvex-recourse-reform} on the disjunct
\begin{equation*}
    \left\lbrace \bs{x} \middle\vert
    \begin{array}{l}
        (\bs{x} - \hat{\bs{x}})^{\top} \bs{l}_{\omega,j} \ge 0, \; j \in \Cs{J}^{+}_{\omega}, \\
        (\bs{x} - \hat{\bs{x}})^{\top} \bs{l}_{\omega,j} \le 0, \; j \in \Cs{J}^{-}_{\omega} \\
    \end{array}
    \right\rbrace,
\end{equation*}
where $\Cs{J}^{+}_{\omega}:=\{j : \hat{y}_{\omega,j}= \underline{y}_{\omega,j}\}$ and $\Cs{J}^{-}_{\omega}:=\{j : \hat{y}_{\omega,j}= \overline{y}_{\omega,j}\}$. 
Let $\Cs{Q}$ be a finite index set, enumerating all combinations of such lower and upper bounds on variables $\bs{y}$. 
Appending the subscript $q \in \Cs{Q}$ to $\hat{\bs{y}}_{\omega}$, $\Cs{J}^{+}_{\omega}$, and $\Cs{J}^{-}_{\omega}$, we define 
\begin{equation*}
    \Cs{K}_{\omega,q}:= \left\lbrace (\bs{x}, \bs{\lambda}, \theta_{\omega}) \middle\vert
    \begin{array}{l}
        \bs{A} \bs{x} \ge \bs{d},  \; \bs{x}  \ge \bs{0}, \; \bs{\lambda} \ge \bs{0}, \\ 
        \theta_{\omega} \ge - \hat{\sigma}   
         \bs{\lambda}^{\top}  \bs{g}(\bs{\xi}_{\omega})  +  \hat{\bs{\pi}}^{\top} (\bs{B}_{\omega}\bs{x}+\bs{b}_{\omega}) + \hat{\sigma}  (\bs{x} - \hat{\bs{x}})^{\top} \bs{L}_{\omega}\hat{\bs{y}}_{\omega,q},\\
        (\bs{x} - \hat{\bs{x}})^{\top} \bs{l}_{\omega,j} \ge 0, \; j \in \Cs{J}^{+}_{\omega,q}, \; 
        (\bs{x} - \hat{\bs{x}})^{\top} \bs{l}_{\omega,j} \le 0, \; j \in \Cs{J}^{-}_{\omega,q} \\
    \end{array}
    \right\rbrace. 
\end{equation*}
As by Lemma \ref{lem: optimality_conditions}, the Lagrangian dual function \eqref{eq: Lagrangian-nonconvex} attains its optimal solution at $\hat{\bs{y}}_{\omega,q}$ for some $q \in \Cs{Q}$, the feasible region of \eqref{eq: DD-nonconvex-recourse-reform}  is a subset of the disjunctive set  
    $\Cs{T}_{\omega}:=\bigcup_{q \in \Cs{Q}} \Cs{K}_{\omega,q}$. 

Consequently, a valid inequality for $\conv{\Cs{T}_{\omega}}$ may be represented in the form \cite[Theorem~3.1]{balas1998disjunctive}
\begin{equation}
    \label{eq: disjunctive-cut}
    \bs{\alpha}_{\bs{x},\omega}^{\top} \bs{x} +\bs{\alpha}_{\bs{\lambda},\omega}^{\top} \bs{\lambda} + \alpha_{\theta,\omega} \theta_{\omega} \ge \rho_{\omega}, 
\end{equation}
where $(\bs{\alpha}_{\bs{x},\omega}, \bs{\alpha}_{\bs{\lambda},\omega}, \alpha_{\theta,\omega}, \rho_{\omega})$ is an element of the reverse polar cone of $\conv{\Cs{T}_{\omega}}$, i.e., the cone characterizing all valid inequalities for $\conv{\Cs{T}_{\omega}}$, denoted by $\Cs{W}_{\omega}$ (see Section~\ref{sec: EC_disjunctive} for its description). Moreover, \eqref{eq: disjunctive-cut} is a facet of $\conv{\Cs{T}_{\omega}}$ only if $(\bs{\alpha}_{\bs{x},\omega}, \bs{\alpha}_{\bs{\lambda},\omega}, \alpha_{\theta,\omega}, \rho_{\omega})$ is an extreme ray of $\Cs{W}_{\omega}$. 
Such a facet can be identified by solving the following cut-generation LP (CGLP): 
\begin{equation}
\label{eq: CGLP}
    \min_{(\bs{\alpha}_{\bs{x},\omega}, \bs{\alpha}_{\bs{\lambda},\omega}, \alpha_{\theta,\omega}, \rho_{\omega}) \in \Cs{W}_{\omega}} \ \bs{\alpha}_{\bs{x},\omega}^{\top} \hat{\bs{x}} +\bs{\alpha}_{\bs{\lambda},\omega}^{\top} \hat{\bs{\lambda}} + \alpha_{\theta,\omega}^{\top} \hat{\theta}_{\omega} - \rho_{\omega}.
\end{equation}
If the optimal value of \eqref{eq: CGLP} is negative and $(\bs{\alpha}_{\bs{x},\omega}, \bs{\alpha}_{\bs{\lambda},\omega}, \alpha_{\theta,\omega}, \rho_{\omega})$ is an optimal solution, then a disjunctive inequality in the form of \eqref{eq: disjunctive-cut} is valid for $\conv{\Cs{T}_{\omega}}$, which cuts off $(\hat{\bs{x}},\hat{\bs{\lambda}}, \hat{\theta}_{\omega})$. 
Otherwise, $(\hat{\bs{x}},\hat{\bs{\lambda}}, \hat{\theta}_{\omega}) \in \conv{\Cs{T}_{\omega}}$. This separation routine is described in Algorithm \ref{alg: SepCuts}, with further details relegated to Section~\ref{sec: EC_disjunctive}. 

\begin{algorithm}[!tb] 
    \SetAlgoLined
        \KwIn{$(\hat{\bs{x}}, \hat{\bs{\lambda}}, \hat{\theta}_{\omega})$ and $(\hat{\bs{\pi}},\hat{\sigma})$.}
	    \KwOut{$(\texttt{viol}, \bs{\alpha}_{\bs{x},\omega}, \bs{\alpha}_{\bs{\lambda},\omega}, \alpha_{\theta,\omega}, \rho_{\omega})$. If a  valid inequality $\bs{\alpha}_{\bs{x},\omega}^{\top} \bs{x} +\bs{\alpha}_{\bs{\lambda},\omega}^{\top} \bs{\lambda} + \alpha_{\theta,\omega} \theta_{\omega} \ge \rho_{\omega}$ is found that is violated by $(\hat{\bs{x}}, \hat{\bs{\lambda}}, \hat{\theta}_{\omega})$, then return \texttt{viol}$=$TRUE, 
	    $\bs{\alpha}_{\bs{x},\omega}$, $\bs{\alpha}_{\bs{\lambda},\omega}$, $\alpha_{\theta,\omega}$, and $\rho_{\omega}$.
	Otherwise, return \texttt{viol}$=$FALSE, $\bs{\alpha}_{\bs{x},\omega}=\bs{0}$, $\bs{\alpha}_{\bs{\lambda},\omega}=\bs{0}$, $\alpha_{\theta,\omega}=0$, and $\rho_{\omega}=0$. }
	
		Initialization: \texttt{viol} $\gets$ FALSE, $\bs{\alpha}_{\bs{x},\omega}=\bs{0}$, $\bs{\alpha}_{\bs{\lambda},\omega}=\bs{0}$, $\alpha_{\theta,\omega}=0$, and $\rho_{\omega}=0$. 
       
	Given $(\hat{\bs{\pi}},\hat{\sigma})$, form $\Cs{W}_{\omega}$ as in \eqref{eq: projectioncone}, with a proper a normalization constraint.

    Let $(\bs{\alpha}^*_{\bs{x},\omega}, \bs{\alpha}^*_{\bs{\lambda},\omega}, \alpha^*_{\theta,\omega}, \rho^{*}_{\omega})$ be an optimal solution to the CGLP $\hat{z}=\min\limits_{(\bs{\alpha}_{\bs{x},\omega}, \bs{\alpha}_{\bs{\lambda},\omega}, \alpha_{\theta,\omega}, \rho_{\omega}) \in \Cs{W}_{\omega}} \ \bs{\alpha}_{\bs{x},\omega}^{\top} \hat{\bs{x}} +\bs{\alpha}_{\bs{\lambda},\omega}^{\top} \hat{\bs{\lambda}} + \alpha_{\theta,\omega}^{\top} \hat{\theta}_{\omega} - \rho_{\omega}$.

    \If{$\hat{z}<0$}{
                \texttt{viol} $\gets$ TRUE and let $\bs{\alpha}_{\bs{x},\omega}=\bs{\alpha}^*_{\bs{x},\omega}$, $\bs{\alpha}_{\bs{\lambda},\omega}=\bs{\alpha}^*_{\bs{\lambda},\omega}$, $\alpha_{\theta,\omega}=\alpha^*_{\theta,\omega}$, and $\rho_{\omega}=\rho^{*}_{\omega}$.
        }

		\caption{\texttt{SepCuts}$(\hat{\bs{x}}, \hat{\bs{\lambda}}, \hat{\theta}_{\omega}; \hat{\bs{\pi}},\hat{\sigma})$.}
  \label{alg: SepCuts}
\end{algorithm}

\begin{theorem}
    \label{thm: DD-finite-convergence-nonconvex}
    Suppose that Assumption 
    {\bf (A4)} holds. 
    For $\omega \in [N]$, suppose that  $h(\bs{x}, \bs{\xi}_{\omega})$ is a nonconvex recourse function defined in \eqref{eq: recourse}, via objective function as in \eqref{eq: nonconvex-recourse}. 
    Then, Algorithm \ref{alg: DD-nonconvex-recourse} generates an $\epsilon$-optimal solution to  \eqref{eq: DRO_Obj}   in a finite number of iterations. 
\end{theorem}

\begin{remark}
    When integer variables are present in the first stage, i.e., $n_{1}>0$, a branch-and-cut procedure can be employed, where a linear relaxation of the master problem is solved using a standard branch-and-bound algorithm. Then, a globally valid optimality cut is obtained whenever an integer feasible solution to the restricted master problem violates any of the previously generated optimality cuts. By the finiteness of the branch-and-bound, it is straightforward to prove the finite convergence of this adjusted version of Algorithm \ref{alg: DD-nonconvex-recourse}. 
Moreover, while Algorithm \ref{alg: DD-nonconvex-recourse} is presented for a DRO problem with a decision-dependent ambiguity set, it can be used to solve DRO problems with a decision-independent ambiguity set, i.e., $\vartheta(\bs{x})=\bs{\vartheta}$. In this case, problem \eqref{eq: reform-recourse} reduces to a two-stage stochastic MIP. 
Thus, Algorithm \ref{alg: DD-nonconvex-recourse} needs to be modified by just changing $\vartheta(\bs{x})$ to $\bs{\vartheta}$. \hfill	$\blacksquare$
\end{remark}

\begin{remark}
When the recourse function $h(\cdot,\bs{\xi})$, defined  in \eqref{eq: recourse}, is a convex function, e.g., when $\psi_{0}(\bs{x}, \bs{y}, \bs{\xi})=\bs{q}^{\top}\bs{y}$, one can obtain Benders-type cut for $G_{\omega}(\bs{x},\bs{\lambda})$, defined in \eqref{eq: Q_k}. More precisely, it can be shown that $G_{\omega}(\bs{x},\bs{\lambda})$ is convex on $(\bs{x},\bs{\lambda})$ for $\omega \in [N]$. Consequently, by exploiting the convex structure of $G_{\omega}(\bs{x},\bs{\lambda})$, one can iteratively obtain its outer approximations using subgradient inequality. 
As a result, the standard finite-convergence argument of two-stage SPs, combined with Assumption {\bf (A4)}, leads to convergence to an $\epsilon$-optimal solution to  \eqref{eq: DRO_Obj}  in a finite number of iterations. We skip the details for brevity. \hfill	$\blacksquare$
\end{remark}

We end this section with a remark on how Algorithm \ref{alg: DD-nonconvex-recourse} is of independent interest to 
solve general two-stage SPs when the cost vector---or equivalently, the recourse matrix---in the second-stage problem is simultaneously (linearly) bi-parameterized by the first-stage decision variables and uncertain parameters. 
\begin{remark}
\label{rem: two-stage}
Following similar ideas as those presented to develop Algorithm \ref{alg: DD-nonconvex-recourse}, a modified algorithm may be developed to find an $\epsilon$-optimal solution to a two-stage stochastic MIP with a nonconvex recourse as:
\begin{equation*}
\min_{\bs{x} \in \Cs{X} } \ \ep{\bs{p}}{h(\bs{x},\bs{\xi})},
\end{equation*} 
where 
    $h(\bs{x},\bs{\xi}_{\omega})= \min_{\bs{y}} \{  (\bs{q}_{\omega} + \bs{L}_{\omega}^{\top}  \bs{x} )^{\top}  \bs{y} : \bs{D}_{\omega}\bs{y} \ge    \bs{B}_{\omega}\bs{x} + \bs{b}_{\omega}, \; 
    \bs{y} \ge \bs{0} \}$,
or equivalently, 
\begin{equation}
\label{eq: nonconvex}
\begin{aligned}
    h(\bs{x},\bs{\xi}_{\omega})= \min_{\bs{y}, \eta} \{\eta :
            \eta -(\bs{q}_{\omega} + \bs{L}_{\omega}^{\top}  \bs{x} )^{\top}  \bs{y} \ge  \bs{0}, 
            \; \bs{D}_{\omega}\bs{y} \ge    \bs{B}_{\omega}\bs{x} + \bs{b}_{\omega}, \; 
            \bs{y} \ge \bs{0} \}. 
\end{aligned}
\end{equation}
Note that $G_{\omega}(\bs{x},\bs{\lambda})$, defined in \eqref{eq: Q_k_nonconvex-recourse}, has the same structure as \eqref{eq: nonconvex}. 
We skip the details of the modified algorithm for brevity. 
\hfill	$\blacksquare$
\end{remark}


\section{Numerical Experiments}
\label{sec: numerical}

We assess the quality of the resulting DRO solutions and provide 
computational comparative results to test the efficacy of our solution methods against solving the extensive formulation using a commercial nonconvex solver. 
All algorithms were implemented in Python 3.12 and solved by GUROBI 12.0.0 as a nonconvex solver. 
All experiments were conducted on a Linux Ubuntu 20.04 system running on a PC equipped with an Intel Core i7-9700 processor at 3.00 GHz and 32 GB of RAM. All codes are available on \url{https://github.com/hamedrahimian/DDRO}. 
In this section, we present the results for the multiproduct newsvendor problem, introduced in Section \ref{sec: NV_app}, while the results for the facility location problem are relegated to \ref{sec: FL}. 

\subsection{Experiment Design}

Consider the newsvendor problem introduced in Section \ref{sec: NV_app}. Note that $h(\bs{q}, \bs{r}, \bs{\xi})$ is in the form of problem \eqref{eq: recourse} with an objective function in the form of \eqref{eq: nonconvex-recourse} upon excluding first-stage costs, yielding a nonconvex recourse function. Hence, Algorithm \ref{alg: DD-nonconvex-recourse} is applicable to solve this problem. 
We assume that for $i \in [n]$, (i) $c_i - g_i >0$ and (ii) $c_{i} < \underline{r}_i$, and (iii) $b_i + \underline{r}_i > c_i $. We note that (i) and (iii) ensure that the critical ratio $0 < \beta_i= \frac{c_i - g_i}{b_i - g_i + \underline{r}_{i}}<1$ is well defined. 
In light of Lemma \ref{lem: boundedness_dual}, we have 
$\Cs{L}=\{\bs{\lambda} : 0 \le \lambda_\iota \le  
\max\{( \overline{r} - \underline{g}) (\overline{q} -\underline{\xi}), \overline{b}(\overline{\xi} - \underline{q})\}, \; \iota \in [s] \}$, where 
$\underline{g}=\min_{ i \in [n]} g_i$, $\overline{r}=\max_{ i \in [n]} r_i$, 
$\overline{b}=\max_{ i \in [n]} b_i$, $\overline{\xi}=\max_{ i \in [n], \omega \in [N]} \xi_i^{\omega}$, $\underline{\xi}=\min_{ i \in [n], \omega \in [N]} \xi_i^{\omega}$, $\overline{q}=\max_{ i \in [n]} d/c_i$, and $\underline{q}=\min_{ i \in [n]} d/c_i$. 
Moreover, in light of the discussion in Section \ref{sec: alg_nonconvex}, variables $y_i^+$ and $y_i^-$ are bounded, i.e., $0 \le y_i^+ \le q_i \le \overline{q}$ and $0 \le y_i^- \le \xi_i$. Finally, due to the bilinear term $r_i y_i^+$ in the objective of \eqref{eq: NV_recourse}, an optimal solution of the Lagrangian dual function in the form of \eqref{eq: Lagrangian-nonconvex} only depends on $y_i^+$.

To conduct experiments, we first generated  $\bs{t}=(\bs{c}^{\top}, \bs{g}^{\top}, \bs{b}^{\top}, \underline{\bs{r}}^{\top}, \overline{\bs{r}}^{\top})^{\top}$, where $c \sim \textrm{Uniform}(0.14,0.34)$, $g \sim \textrm{Uniform}(0.02,0.12)$, $b \sim \textrm{Uniform}(0.18,0.38)$, 
$\underline{r} \sim \textrm{Uniform}(0.4,0.6)$, and $\overline{r} \sim \textrm{Uniform}(0.7,0.9)$. 
For $j \in [n]$, we formed $w_{ij}:=\exp{(-\|t_i-t_j\|}$ for $i \neq j$ and $w_{ij}:=-0.1$ for $i=j$, normalized these weights over $i \in [n]$ to obtain $\tilde{w}^{\mu}_{ij}$. 
We set $u^\mu_{ij}=\rho_{\mu} \tilde{w}^{\mu}_{ij}$.
Next, we calculated $\tilde{w}^\sigma_{ij}=\tilde{w}^\mu_{ij}-\frac{1}{2}\min_{k \in [n]} \tilde{w}^\mu_{kj}$, $i\in[n]$. We then set $u^\sigma_{ij}=\rho_{\sigma} \tilde{w}^{\sigma}_{ij}$. Parameters  $\rho_{\mu}$ and $\rho_{\sigma}$ are tunable and control the {\it degree} of decision-dependency, where $\rho_{\sigma} \le 1$ so that $\sum_{i \in [n]}u^\sigma_{ij}<1$. 
By construction, $w_{ij}$ measures the similarity or substitutability of products. This design signifies that increasing prices of more {\it similar} alternative products increases the mean demand, while decreasing its noise. 

For each training set, we chose the vector of nominal mean of the random demand $\bs{\xi}$, $\overline{\bs{\mu}}$, independently from  $\text{Uniform}(10,20)$ and set $\overline{\bs{\sigma}}=\gamma \overline{\bs{\mu}}$, where $\gamma$ is a tunable parameter to control the coefficient of variation.
Realizations of the random demand $\xi_j$, $j \in [m]$, were generated independently from a folded normal distribution with mean $\overline{\mu}_j$ and standard deviation $\overline{\sigma}_j$.

\subsection{Computational Results}
In this section, we compare the computational performance of Algorithm \ref{alg: DD-nonconvex-recourse}, denoted as \texttt{DECOMPOSED}, against the MINLP deterministic equivalent formulation, presented in \eqref{eq: reform-simple-moment-recourse-linearized}, which is solved using an off-the-shelf nonconvex solver and denoted as \texttt{DEF}.
We set the time limit to 3600 seconds.  
We present the computational results for $n \in \{2,3\}$ products, $d= n/2 $, hyperparameters $\tau_\mu=\underline{\tau}_\sigma=0$, $\overline{\tau}_\sigma =1 + \tau_{\sigma}$, where $\tau_{\sigma} \in \{0,1\}$ for the ambiguity set, $\rho_{\mu}=\rho_{\sigma} =1$ for the degree of decision dependency, $\gamma=1$ for the coefficient of variation, and $N \in \{500,1000,2000,3000,5000\}$. 

\begin{table}[!ht]
\centering
\scriptsize
\caption{Comparison of \texttt{DEF} and \texttt{DECOMPOSED} to solve the newsvendor problem with a price-dependent, $\texttt{DD}$, ambiguity set, with $\gamma=1$,  $\rho_{\mu}=\rho_{\sigma}=1$, and $\tau_\mu=\underline{\tau}_\sigma=0$.}
\label{T: NV_Comp_A}
\begin{tabular}{lllllll}
\toprule
  &    & & \multicolumn{2}{c}{\texttt{DECOMPOSED}} & \multicolumn{2}{c}{\texttt{DEF}} \\
  \cmidrule(lr){4-5} \cmidrule(lr){6-7} 
$\overline{\tau}_\sigma$  & $(n,d)$ & $N$    &  Gap (\%) &    Time (s) & Gap (\%) &    Time (s) \\
\midrule
1.0            &  (2,1) & 500  &        0.0 &        17.78 &          0.01 &        76.06 \\
           &   & 1000 &        0.0 &        80.41 &          0.01 &        51.84 \\
           &   & 2000 &        0.0 &        225.8 &          0.01 &       136.89 \\
           &   & 3000 &        0.0 &        254.1 &          0.01 &        281.3 \\
           &   & 5000 &        0.0 &        589.7 &          0.01 &       632.25 \\
\cmidrule{2-7}          
           & (3,1.5)  & 500  &        0.0 &        74.54 &          0.01 &       469.85 \\
           &   & 1000 &        0.0 &       234.15 &          0.02 &      1632.64 \\
           &   & 2000 &        0.0 &       554.82 &  0.01 (21.15) &  1825.09 (3) \\
           &   & 3000 &        0.0 &       929.31 &    - (8.64) &      - (5) \\
           &   & 5000 &  0.0  &   1469.6 &   - (27.79) &      - (5) \\
\midrule          
2.0            & (2,1)  & 500  &        0.0 &        16.99 &          0.01 &         8.62 \\
           &   & 1000 &        0.0 &        35.47 &          0.01 &        27.06 \\
           &   & 2000 &        0.0 &        115.7 &          0.01 &       101.71 \\
           &   & 3000 &        0.0 &       383.49 &          0.02 &       283.51 \\
           &   & 5000 &        0.0 &       281.05 &          0.02 &       414.13 \\
\cmidrule{2-7}          
           & (3,1.5)  & 500  &        0.0 &        72.03 &          0.05 &        503.4 \\
           &   & 1000 &        0.0 &       201.25 &          0.07 &      1654.14 \\
           &   & 2000 &        0.0 &       384.08 &  0.06 (22.12) &  1139.24 (2) \\
           &   & 3000 &        0.0 &       768.88 &  0.09 (17.67) &  2084.86 (3) \\
           &   & 5000 &  0.0 &  1351.46 &   - (28.58) &      - (5) \\
\bottomrule
\end{tabular}

\end{table}

Table \ref{T: NV_Comp_A} reports the average computational results (over five training sets) for DRO with decision-dependent ambiguity set, $\texttt{DD}$. 
The values under the column ``Gap (\%)'' indicate the average gap (in \%) for instances that could be solved optimally. In parentheses, the average gap for instances that could not be solved optimally within the time limit is also provided. 
Additionally, the values under the column ``Time (s)'' indicate the average time (in seconds) required for instances that could be solved optimally. In parentheses, the number of instances (out of five) that could not be solved optimally within the time limit is also provided. 
Observe from Table \ref{T: NV_Comp_A} that $\texttt{DECOMPOSED}$ found an optimal solution within the time limit for all instances. 
Whereas $\texttt{DEF}$ stopped with a nonzero optimality gap in some cases, on average between [8.64\%-28.58\%]. 
In addition, for instances that could be solved optimally within the time limit with both approaches, $\texttt{DECOMPOSED}$  often found an optimal solution with less computational effort. We observe that $\texttt{DEF}$ generally had a longer average runtime and optimality gap for instances with more scenarios and products. Moreover, an increase of $\tau_\sigma$ from 0 to 1---a wider range on the second-order moment---led to generally shorter runtimes for both approaches. 

While we have not provided the detailed computational results for the DRO problem with a price-independent ambiguity set, both $\texttt{DEF}$ and $\texttt{DECOMPOSED}$ were able to solve {\it all} problem instances within sixty seconds (on average). This observation further highlights the increased computational burden associated with solving the price-dependent DRO problem compared to its price-independent counterpart (see the computational performance profile in Section~\ref{sec: EC_NV} for additional insights). 


\subsection{Out-of-Sample Performance}
We also report the out-of-sample performance of solutions resulting from a DRO with decision-dependent, $\texttt{DD}$, and decision-independent ambiguity set, $\texttt{DI}$, for a problem instance with $n=3$ and $d=1.5$, with 
$N=100$, 
$\rho_{\mu}, \rho_{\sigma} \in \{0.3,0.5,0.7\}$ for the degree of decision dependency, and 
$\gamma \in \{0.5,0.7\}$ 
for the coefficient of variation. 
To form the ambiguity sets, we used hyperparameters $\tau_\mu=0.5$, $\underline{\tau}_\sigma=0$, and $\overline{\tau}_\sigma=1+ \tau_{\sigma}$, where $\tau_{\sigma} \in \{0, 0.2, 0.4\}$.

To report the results, we performed 25 microsimulations, i.e., 25 sets of training data of size $N$ were generated. 
To calculate the out-of-sample costs (of each pair of solutions), we used 30 common test sets of size 500, drawn from a standard normal distribution. 
Given a solution $\bs{x}=(\bs{q},\bs{r})$ (resulting from a training set), its corresponding test realizations were generated following a folded normal transformation, where the mean and variance were calculated via \eqref{NV_moments}, and the nominal mean and standard deviation were set to those of the training set.  We then calculated a normalized upper bound of a two-sided 95\% confidence interval (UCB) on the mean cost difference $(\texttt{DI} - \texttt{DD})$, where the normalization was calculated based on $100 / \textrm{mean of out-of-sample } |\texttt{DI}| \%$. 

Figs.~\ref{fig: NV_oos_0.5_100} and \ref{fig: NV_oos_0.7_100} report a two-sided 95\% confidence interval of these UCBs over 25 microsimulations. 
Observe that the out-of-sample cost difference $(\texttt{DI} - \texttt{DD})$ is positive, indicating that $\texttt{DD}$ consistently yielded a lower out-of-sample cost than \texttt{DI}. 

\paragraph{Sensitivity with respect to $\rho_{\mu}$ and $\rho_{\sigma}$.}
In Fig.~\ref{fig: NV_oos_0.5_100}, at a fixed $\rho_{\sigma}$, the slope of each curve captures the sensitivity to stronger price-dependent mean shifts through $\rho_{\mu}$, while the vertical separation between curves at a fixed $\rho_{\mu}$ captures the incremental effect of stronger price-dependent dispersion through $\rho_{\sigma}$.
Observe that, for each value of $\rho_\sigma$, the curves are increasing in $\rho_{\mu}$. Thus, the gap $(\texttt{DI}-\texttt{DD})$ becomes larger as the degree of decision dependence in the mean increases. This is consistent with the newsvendor model: increasing $\rho_{\mu}$ strengthens the substitution-driven effect of prices on demand means, so the price-independent model becomes increasingly misspecified when it evaluates stocking and pricing decisions. In contrast, \texttt{DD} internalizes how its price decisions shift the demand distribution and can better coordinate prices and stocking quantities under the induced demand pattern.
By comparison, the curves corresponding to different values of $\rho_{\sigma}$ are nearly overlapping in most panels. This indicates that, for these instances, the out-of-sample gap is much less sensitive to $\rho_{\sigma}$ than to $\rho_{\mu}$. One reason is that $\rho_{\sigma}$ affects the dispersion channel, which primarily changes tail-risk protection, whereas the revenue and shortage trade-offs in the newsvendor objective are directly affected by mean demand shifts. Moreover, since the price-dependent variance in \eqref{NV_moments} decreases with higher prices of substitutable products, changing $\rho_{\sigma}$ often moves both models in a similar conservative direction; hence, the marginal value of modeling price-dependent dispersion is smaller than the marginal value of modeling price-dependent mean demand.

\begin{figure}[!tb]
	\centering
    \includegraphics[width=0.7\linewidth]{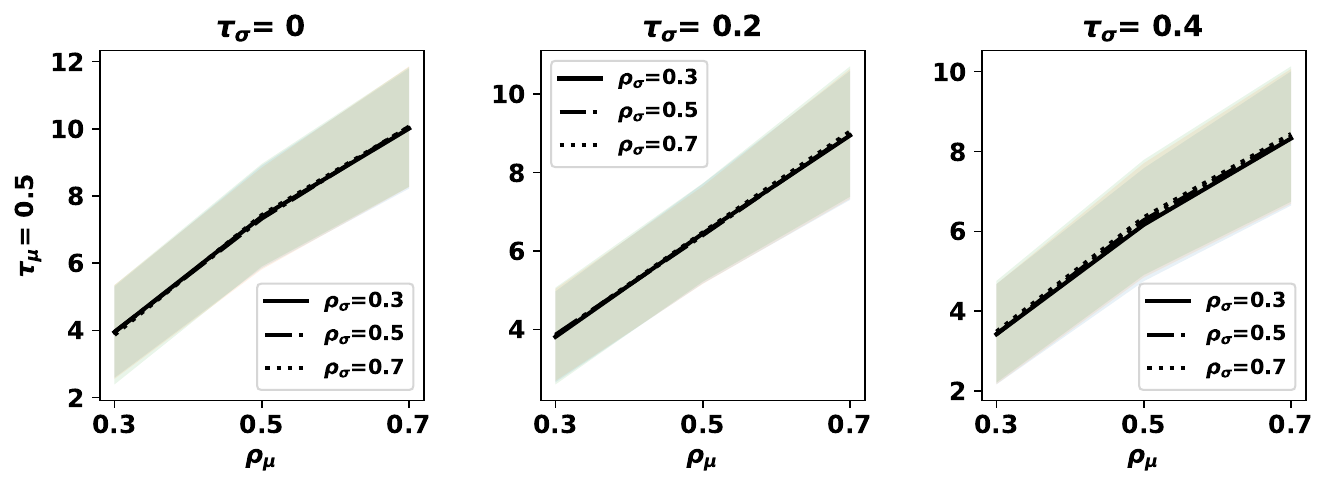}
	\caption{\label{fig: NV_oos_0.5_100} Average 95\% UCB on the mean out-of-sample difference $(\texttt{DI} - \texttt{DD})$ for the newsvendor problem with $n=3$, $d=1.5$, $\underline{\tau}_\sigma=0$, $\gamma=0.5$, and $N=100$.}
\end{figure}

\begin{figure}[!tb]
	\centering
    \includegraphics[width=0.7\linewidth]{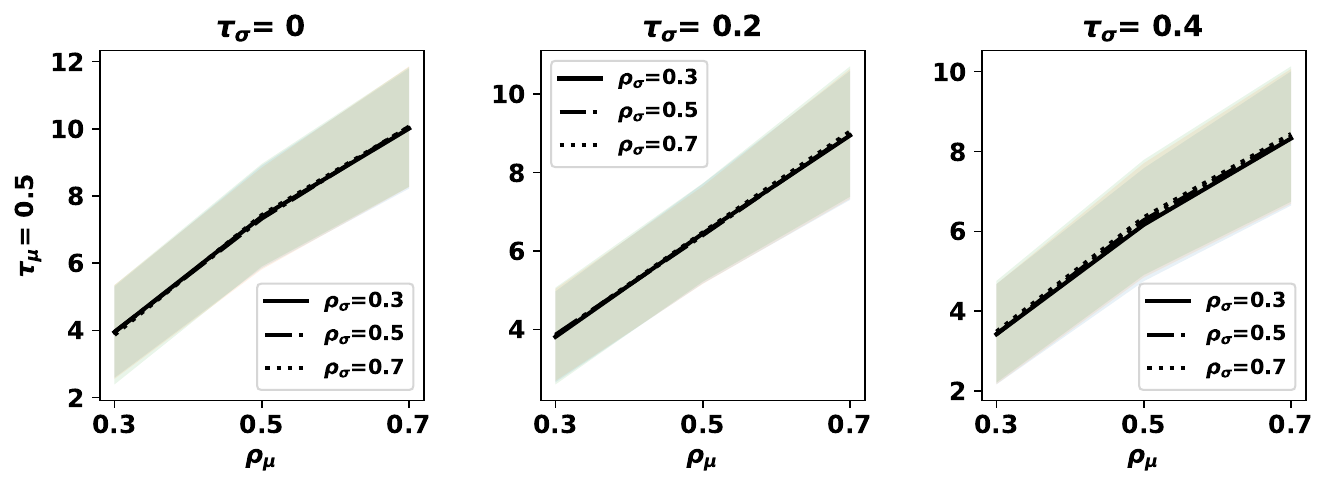}
	\caption{\label{fig: NV_oos_0.7_100} Average 95\% UCB on the mean out-of-sample difference $(\texttt{DI} - \texttt{DD})$ for the newsvendor problem with $n=3$, $d=1.5$, $\underline{\tau}_\sigma=0$, $\gamma=0.7$, and $N=100$.}
\end{figure}

\paragraph{Sensitivity with respect to $\tau_\sigma$.} 
Fig.~\ref{fig: NV_oos_0.5_100} isolates the effect of the upper second-moment tolerance. Recall that $\tau_\sigma=0$ imposes no upper deviation from the nominal second moment, whereas $\tau_\sigma=0.2$ and $\tau_\sigma=0.4$ allow increasingly larger upper deviations and therefore more tail-risk ambiguity. Across the panels, the positive gap $(\texttt{DI}-\texttt{DD})$ persists, showing that the benefit of incorporating price-dependent ambiguity is robust to $\tau_\sigma$.
At the same time, increasing $\tau_\sigma$ does not materially change the qualitative pattern with respect to $\rho_{\mu}$ and $\rho_{\sigma}$: the curves remain upward sloping in $\rho_{\mu}$ and remain close across $\rho_{\sigma}$. This suggests that the main source of out-of-sample improvement is the ability of \texttt{DD} to track price-dependent mean demand rather than the variance channel alone. A larger $\tau_\sigma$ can make both \texttt{DD} and \texttt{DI} hedge more against high-demand tail outcomes, which may weaken the incremental value of decision-dependent dispersion. Nevertheless, because price-induced mean misspecification remains important, \texttt{DD} continues to outperform \texttt{DI} across the tested values of $\tau_\sigma$.

Figs.~\ref{fig: NV_oos_obj_0.5_100} and \ref{fig: NV_oos_obj_0.7_100} report a two-sided 95\% confidence interval on the {\it postdecision disappointment} ``out-of-sample -  in-sample,'' calculated by the difference of corresponding values of  $(\texttt{DI} - \texttt{DD})$. Observe that the postdecision disappointment difference $(\texttt{DI} - \texttt{DD})$ is consistently positive. 
Recall that DRO models are typically associated with (on average) a positive postdecision disappointment, because they often yield optimistic in-sample costs that may not realize out of sample, though still outperforming classical SP models. 
Our observations confirm this expected behavior as both \texttt{DD} and \texttt{DI} exhibit a positive postdecision disappointment. However, the decision-dependent ambiguity substantially mitigates the optimizer's curse compared to decision-independent DRO, leading to a smaller postdecision disappointment. This improvement, however, was achieved at the expense of increased computational effort, as reported in Table~\ref{T: NV_Comp_A}.

\begin{figure}[!tb]
	\centering
    \includegraphics[width=0.7\linewidth]{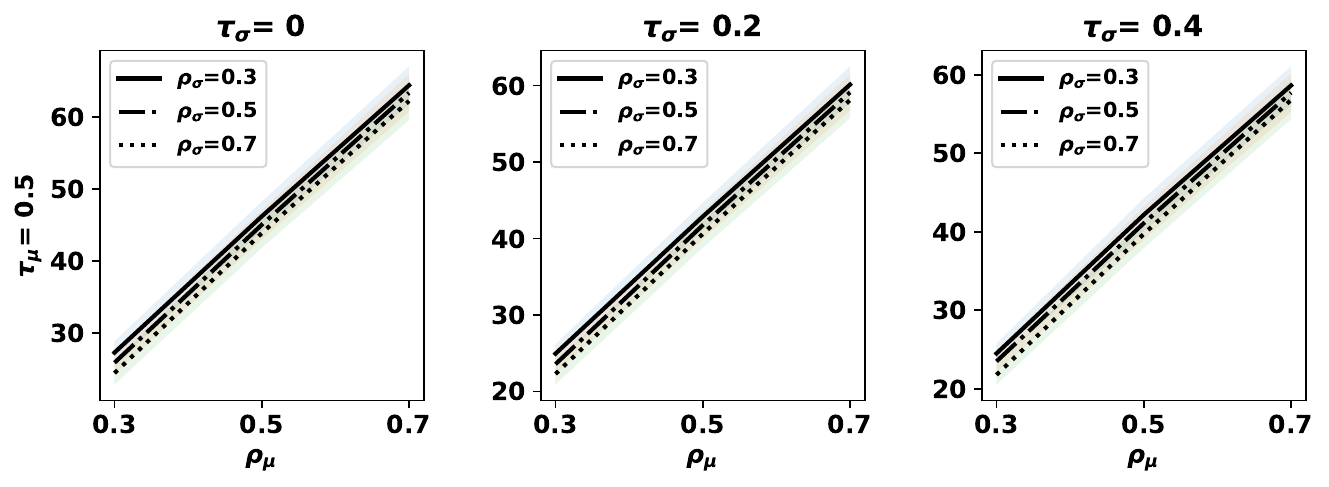}
	\caption{\label{fig: NV_oos_obj_0.5_100} 95\% UCB on the mean postdecision disappointment difference $(\texttt{DI} - \texttt{DD})$ for the newsvendor problem with $n=3$, $d=1.5$, $\underline{\tau}_\sigma=0$, $\gamma=0.5$, and $N=100$.}
\end{figure}	

\begin{figure}[!tb]
	\centering
    \includegraphics[width=0.7\linewidth]{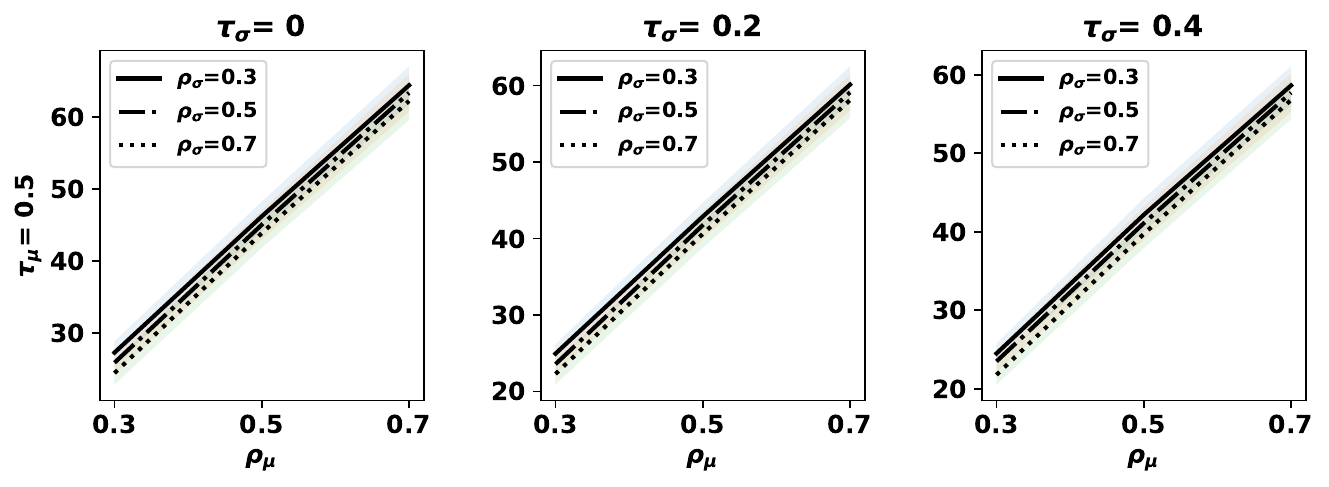}
	\caption{\label{fig: NV_oos_obj_0.7_100} 95\% UCB on the mean postdecision disappointment difference $(\texttt{DI} - \texttt{DD})$ for the newsvendor problem with $n=3$, $d=1.5$, $\underline{\tau}_\sigma=0$, $\gamma=0.7$, and $N=100$.}
\end{figure}

\section{Conclusion}
\label{sec: discussion}

Despite theoretical and computational advances in decision-dependent DRO, there is a limited understanding of the importance of incorporating decision-dependent/adaptive distributional ambiguity.
To reveal the intrinsic balance between modeling fidelity, solution quality, and computational
efficiency that characterizes the decision-dependent DRO framework, we considered a general two-stage stochastic mixed-integer program with (non)convex continuous recourse. We assumed that the probability distribution of random parameters is unknown and depends on decisions. We thus investigated a distributionally robust approach to this problem, where the distributional ambiguity is modeled with a polyhedral decision-dependent ambiguity set. Motivated by representative applications in joint pricing-stocking problems with price-dependent demand and facility location problems with location-dependent demand, we considered cases where the objective function of the recourse is bilinear or linear in the first- and second-stage decision variables, leading to a nonconvex or convex recourse, respectively. Leveraging Lagrangian/linear programming duality, we reformulated the resulting problem as a two-stage stochastic mixed-integer nonlinear program. Building upon this reformulation, we proposed decomposition-based cutting-plane algorithms to obtain an $\epsilon$-optimal solution to the resulting problems and established their finite convergence. The proposed algorithm for the nonconvex recourse may be of independent interest in solving two-stage stochastic programs  when the cost vector---or equivalently, the recourse matrix---in the second-stage problem is simultaneously (linearly) bi-parameterized by the first-stage decision variables and uncertain parameters.
Our experiments on benchmark instances indicate that a decision-dependent ambiguity set substantially reduces the out-of-sample cost (up to 11\% and 7\% on average, respectively for the newsvendor and facility location problems), 
thereby mitigating the optimizer's curse relative to a decision-independent DRO, up to 65\%. However, this improvement is achieved at the expense of an increased computational effort, with the absolute runtime ratio falling between 2 and 25 for the medium- to large-sized newsvendor instances and between 2 and 12 for the facility location instances. These results indicate that decision-dependent DRO problems are inherently more challenging to solve than their decision-independent counterparts. Nevertheless, our proposed solution algorithms outperform the direct solution of the extensive formulation using a commercial nonconvex solver.

Future research will examine the theoretical properties of the gap between postdecision disappointment in a decision-dependent DRO and that of a decision-independent DRO, both evaluated under full knowledge of the true decision-dependent distribution.
From a computational perspective, we investigate the case where the sample space is infinite. As in the DRO literature with a decision-independent ambiguity set, more generalized forms of duality, e.g., conic duality, are expected to be needed for reformulating the problem as a semi-infinite program, where a delayed cut generation scheme may be used to find a worst-case realization. Especially for the nonconvex recourse, it would be interesting to explore how the proposed disjunctive cutting-plane algorithm may be extended. Another direction for future research is to investigate decision-dependent stochastic programs with probabilistic constraints.


\ECSwitch 


\ECHead{Electronic Companion}
\section{Proofs for Problem Reformulation}

\proof{Proof of Theorem \ref{thm: simple-moment-inner-lag-dual}}
    Let us define $\Cs{M}:=\Set*{\bs{p} \ge \bs{0}}{ \ul{p}_{\omega} \le  p_{\omega} \le \ol{p}_{\omega}, \; \omega \in [N]}$.
    For a fixed $\bs{x} \in \Cs{X}$, by dualizing the second set of constraints in \eqref{eq: simple-moment}, a Lagrangian function of problem $\max_{\bs{p} \in \Cs{P}(\bs{x})} \ \ep{\bs{p}}{h(\bs{x},\bs{\xi})}$  can be written as: 
        $\Fs{L}(\bs{x}, \bs{p}, \bs{\lambda})= \bs{\lambda}^{\top} \bs{\vartheta}(\bs{x})+ \sum_{\omega \in [N]} p_{\omega} \Big(h(\bs{x},\bs{\xi}_{\omega}) - \bs{\lambda}^{\top}\bs{g}(\bs{\xi}_{\omega})\Big)$.
    Hence, the Lagrangian dual of $\max\limits_{\bs{p} \in \Cs{P}(\bs{x}) }
    \ \ep{\bs{p}}{h(\bs{x},\bs{\xi})}$ is 
     $   \min\limits_{ \bs{\lambda} \ge \bs{0}}   \max\limits_{ \bs{p} \in \Cs{M}}  \ \Fs{L}(\bs{x}, \bs{p}, \bs{\lambda})$.
We have    \begin{align*}
         - & \bs{\lambda}^{\top} \bs{\vartheta}(\bs{x}) +  \max_{ \bs{p} \in \Cs{M}}  \Fs{L}(\bs{x}, \bs{p}, \bs{\lambda}) \\
        =& \sum_{\omega \in [N]} \ol{p}_{\omega} \Big(h(\bs{x},\bs{\xi}_{\omega}) - \bs{\lambda}^{\top} \bs{g}(\bs{\xi}_{\omega})\Big)_{+}   -   \sum_{\omega \in [N]} \ul{p}_{\omega} \Big(-h(\bs{x},\bs{\xi}_{\omega}) + \bs{\lambda}^{\top} \bs{g}(\bs{\xi}_{\omega}) \Big)_{+} \\
        = & \sum_{\omega \in [N]} \ol{p}_{\omega} \Big(h(\bs{x},\bs{\xi}_{\omega}) - \bs{\lambda}^{\top} \bs{g}(\bs{\xi}_{\omega})\Big)_{+}  - \ul{p}_{\omega} \Big(-h(\bs{x},\bs{\xi}_{\omega}) + \bs{\lambda}^{\top} \bs{g}(\bs{\xi}_{\omega}) \Big)_{+} \\
        = &  \sum_{\omega \in [N]} \varphi_{\omega} \Big[ h(\bs{x},\bs{\xi}_{\omega}) - \bs{\lambda}^{\top} \bs{g}(\bs{\xi}_{\omega}) \Big] = \sum_{\omega \in [N]} G_{\omega}(\bs{x}, \bs{\lambda}).  
    \end{align*}
    The second equality above is due to the fact that either $\Big(h(\bs{x},\bs{\xi}_{\omega}) - \bs{\lambda}^{\top} \bs{g}(\bs{\xi}_{\omega})\Big)_{+}$ or $\Big(-h(\bs{x},\bs{\xi}_{\omega}) + \bs{\lambda}^{\top} \bs{g}(\bs{\xi}_{\omega})\Big)_{+}$ is positive. Now, because $\max_{\bs{p} \in \Cs{P}(\bs{x}) }
    \ \ep{\bs{p}}{h(\bs{x},\bs{\xi})}$ is an LP and by Assumption {\bf (A2)}, there is no duality gap. Combining the resulting dual problem with the outer minimization problem yields the reformulation stated in the theorem. \hfill \Halmos
\endproof

We state a lemma, which will be useful to prove Corollary \ref{cor: reform-recourse}.  
\begin{lemma}
    \label{lem: convex_function}
    Consider function $\varphi_{\omega}[z]=    \ol{p}_{\omega} (z)_{+} -   \ul{p}_{\omega} (-z)_{+}$ for $\omega \in [N]$. Then, $\varphi_{\omega}[z]$ is convex and monotonically nondecreasing in $z$. 
\end{lemma}

\proof{Proof of Lemma \ref{lem: convex_function}}
    Observe that $\varphi_{\omega}[z]$ can be rewritten as $\varphi_{\omega}[z]=  \ul{p}_{\omega} z+ \varepsilon (z)_{+}$,
    where $\varepsilon \ge 0$ is such that $\ol{p}_{\omega}=\ul{p}_{\omega}+\varepsilon$. Hence, $\varphi_{\omega}[z]$ is monotonically nondecreasing in $z$. Because $(z)_{+}$ is convex in $z$, then, it follows that $\varphi_{\omega}[z]$ is convex in $z$. \hfill \Halmos
\endproof

\proof{Proof of Corollary \ref{cor: reform-recourse}}
By Lemma \ref{lem: convex_function}, $\varphi_{\omega}[z]$ is monotonically nondecreasing in $z$. Hence, 
\begin{align*}
    G_{\omega}(\bs{x},\bs{\lambda})= 
    \varphi_{\omega} \Big [h(\bs{x},\bs{\xi}_{\omega}) - \bs{\lambda}^{\top} \bs{g}(\bs{\xi}_{\omega})\Big] 
    = \min_{\bs{y} \in \Cs{Y}(\bs{x}, \bs{\xi}_{\omega}) } \varphi_{\omega} \Big [\psi_{0}(\bs{x},\bs{y}, \bs{\xi}_{\omega}) - \bs{\lambda}^{\top} \bs{g}(\bs{\xi}_{\omega})\Big]. 
\end{align*}
Now, linearization of $\varphi_{\omega}[\cdot]$ by introducing additional variables $\gamma$ and $\mu$, yields \eqref{eq: Q_k_recourse}. Theorem \ref{thm: simple-moment-inner-lag-dual} completes the proof. \hfill \Halmos 
\endproof

We end this section by establishing that problem \eqref{eq: reform-simple-moment} admits bounded optimal multipliers $\bs{\lambda}$. These lemmas will be helpful in establishing the finite convergence of Algorithm \ref{alg: DD-nonconvex-recourse}. 

\begin{lemma}
    \label{lem: convex_lambda_function}
    For a fixed $\bs{x} \in \Cs{X}$ and $\omega \in [N]$, 
    $G_{\omega}(\bs{x},\bs{\lambda})$,  defined in \eqref{eq: Q_k}, is a proper, convex, continuous function in $\bs{\lambda}$ on $\Bs{R}^{s}$. 
\end{lemma}

\proof{Proof of Lemma \ref{lem: convex_lambda_function}}
    By the boundedness of $h(\bs{x},\bs{\xi}_{\omega})$ (implied by Assumption {\bf A3}), we have $G_{\omega}(\bs{x},\bs{\lambda}) > -\infty$ for all $\bs{\lambda} \in \Bs{R}^{s}$ and there exists $\bs{\lambda} \in \Bs{R}^{s}$ with $G_{\omega}(\bs{x},\bs{\lambda}) < \infty$, e.g., $\bs{\lambda}=\bs{0}$; proving $G_{\omega}(\bs{x},\cdot)$ is a proper function. 
    In addition, by the convexity of $\varphi_{\omega}[\cdot]$ from Lemma \ref{lem: convex_function} and linearity of $h(\bs{x},\bs{\xi}_{\omega}) - \bs{\lambda}^{\top} \bs{g}(\bs{\xi}_{\omega}) $ in $\bs{\lambda}$, $G_{\omega}(\bs{x},\cdot)$  is a convex function, and hence, continuous on $\Bs{R}^{s}$. \hfill \Halmos
\endproof

\begin{lemma}
    \label{lem: boundedness_dual}
    There exists $\overline{\lambda}<\infty$ such that optimizing over the compact set $\Cs{L}:=\Set*{\bs{\lambda} \in \Bs{R}_{+}^{s}}{0 \le \lambda_i \le \overline{\lambda}, \; i \in [s]}$ does not change the optimal value to \eqref{eq: reform-simple-moment}. 
\end{lemma}

\proof{Proof of Lemma \ref{lem: boundedness_dual}}
    For a fixed $\bs{x} \in \Cs{X}$, following the proof of Theorem \ref{thm: simple-moment-inner-lag-dual}, $\max_{\bs{p} \in \Cs{P}(\bs{x})}  \linebreak  \ep{\bs{p}}{h(\bs{x},\bs{\xi})}$ can be reformulated as the minimization problem $\min_{\bs{\lambda} \ge \bs{0}} F_{\bs{x}}(\bs{\lambda})$, where
        $F_{\bs{x}}(\bs{\lambda}):= \bs{\lambda}^{\top}\bs{\vartheta}(\bs{x})  + \sum_{\omega \in [N]} G_{\omega}(\bs{x}, \bs{\lambda})$.
    Let $\Cs{M}:=\Set*{\bs{p} \in \Bs{R}^{N}}{\ul{p}_{\omega} \le p_{\omega} \le \ol{p}_{\omega}, \; \omega \in [N]}$, $\bs{h}_{\bs{x}}:=(h(\bs{x},\bs{\xi}_{\omega}))_{\omega \in [N]}$, and $\bs{A}:=[\bs{g}(\bs{\xi}_{1}) \ \cdots \ \bs{g}(\bs{\xi}_{N})]$. Then
        $F_{\bs{x}}(\bs{\lambda})=\max_{\bs{p} \in \Cs{M}} \left\{\bs{h}_{\bs{x}}^{\top}\bs{p}+\bs{\lambda}^{\top}\left(\bs{\vartheta}(\bs{x})-\bs{A}\bs{p}\right)\right\}$.
    Suppose $\bs{d}\in\Bs{R}_{+}^{s}$ satisfies
        $\max_{\bs{p}\in\Cs{M}} \ \bs{d}^{\top}\left(\bs{\vartheta}(\bs{x})-\bs{A}\bs{p}\right)=0$.
    Then, for every $t\ge 0$ and every optimal multiplier $\bs{\lambda}$,
    \begin{align*}
        F_{\bs{x}}(\bs{\lambda}+t\bs{d})
        &=\max_{\bs{p} \in \Cs{M}} \left\{\bs{h}_{\bs{x}}^{\top}\bs{p}+\bs{\lambda}^{\top}\left(\bs{\vartheta}(\bs{x})-\bs{A}\bs{p}\right)+t\bs{d}^{\top}\left(\bs{\vartheta}(\bs{x})-\bs{A}\bs{p}\right)\right\} \\
        & = \max_{\bs{p} \in \Cs{M}} \left\{\bs{h}_{\bs{x}}^{\top}\bs{p}+\bs{\lambda}^{\top}\left(\bs{\vartheta}(\bs{x})-\bs{A}\bs{p}\right) \right\} + t \max_{ \bs{p} \in \Cs{M}}\bs{d}^{\top}\left(\bs{\vartheta}(\bs{x})-\bs{A}\bs{p}\right)\\
        & \le F_{\bs{x}}(\bs{\lambda}).
    \end{align*}
    Hence, by optimality of $\bs{\lambda}$, equality holds. Therefore, unbounded rays of optimal multipliers, if present, do not change the objective value, and it suffices to select a bounded representative optimal multiplier.
    To construct such a representative uniformly over $\bs{x}\in\Cs{X}$, consider the equivalent LP dual
    \begin{equation*}
        \min_{\bs{\lambda},\bs{\gamma},\bs{\mu}}\left\{\bs{\lambda}^{\top}\bs{\vartheta}(\bs{x})+\sum_{\omega\in[N]}\ol{p}_{\omega}\gamma_{\omega}-\sum_{\omega\in[N]}\ul{p}_{\omega}\mu_{\omega}:\;
        \bs{A}^{\top}\bs{\lambda}+\bs{\gamma}-\bs{\mu}=\bs{h}_{\bs{x}},\; \bs{\lambda},\bs{\gamma},\bs{\mu}\ge\bs{0}\right\}.
    \end{equation*}
    Eliminating $\bs{\gamma}$ and $\bs{\mu}$ recovers $\min_{\bs{\lambda}\ge\bs{0}}F_{\bs{x}}(\bs{\lambda})$. Since $\Cs{P}(\bs{x})$ is nonempty by Assumption {\bf (A2)} and $\Cs{M}$ is bounded, the primal LP is feasible and bounded; thus, the above dual has a finite optimal extreme point. Every such extreme point is a basic feasible solution of the system with coefficient matrix $[\bs{A}^{\top} \ I \ -I]$ and right-hand side $\bs{h}_{\bs{x}}$. 
    Let $\overline{h}:=\sup\{|h(\bs{x},\bs{\xi}_{\omega})|: \bs{x}\in\Cs{X},\; \omega\in[N]\} < \infty$, where finiteness follows from the boundedness of $h(\bs{x},\bs{\xi}_{\omega})$ (implied by Assumption {\bf A3}) and the compactness of $\Cs{X}$.  
    Because $|h(\bs{x},\bs{\xi}_{\omega})|\le \overline{h}$ for all $\bs{x}\in\Cs{X}$ and $\omega\in[N]$, and because there are finitely many bases of $[\bs{A}^{\top} \ I \ -I]$, there is a constant $\overline{\lambda}<\infty$ that bounds the $\bs{\lambda}$-components of these optimal basic feasible solutions for all $\bs{x}\in\Cs{X}$. Therefore, for every $\bs{x}\in\Cs{X}$, the minimization of $F_{\bs{x}}$ admits an optimal solution in $\Cs{L}$, and restricting $\bs{\lambda}$ to $\Cs{L}$ does not change the optimal value of \eqref{eq: reform-simple-moment}.  \hfill \Halmos
\endproof

\section{Details of Disjunctive Cut and its Separation}
\label{sec: EC_disjunctive}
We have 
\begin{equation}
    \label{eq: projectioncone}
    \Cs{W}_{\omega}:= \left\lbrace  (\bs{\alpha}_{\bs{x},\omega}, \bs{\alpha}_{\bs{\lambda},\omega}, \alpha_{\theta,\omega}, \rho_{\omega}) \middle\vert
    \begin{array}{l}
        \exists \; \{\delta_{\omega, j, q}^{+} \ge 0: j \in \Cs{J}^{+}_{\omega,q} \},
         \{\delta_{\omega,j,q}^{-} \ge 0: j \in \Cs{J}^{-}_{\omega,q}\}, \\
        \quad \delta_{\theta,q} \ge 0, \; \bs{\delta}_{\bs{x},q} \ge \bs{0}, \; q \in \Cs{Q}, \; \st \\ 
        \bs{A}^{\top} \delta_{\bs{x},q}  - \delta_{\theta,q} \hat{\sigma}  \bs{L}_{\omega} \hat{\bs{y}}_{\omega,q} -  \delta_{\theta,q} \bs{B}_{\omega}^{\top} \hat{\bs{\pi}} + \\ 
        \quad \sum\limits_{ j \in \Cs{J}^{+}_{\omega,q} } \delta_{\omega, j, q}^{+} \bs{l}_{\omega,j} - \sum\limits_{ j \in \Cs{J}^{-}_{\omega,q} } \delta_{\omega, j, q}^{-} \bs{l}_{\omega,j}  \le \bs{\alpha}_{\bs{x},\omega}, \; q \in \Cs{Q}, \\
        \delta_{\theta,q} \hat{\sigma}  \bs{g}(\bs{\xi}_{\omega}) \le \bs{\alpha}_{\bs{\lambda},\omega}, \; q \in \Cs{Q}, \\ 
        \delta_{\theta,q} \le \alpha_{\theta,\omega}, \; q \in \Cs{Q}, \\ \bs{\delta}_{\bs{x},q}^{\top} \bs{d} + \delta_{\theta,q} \hat{\bs{\pi}}^{\top} \bs{b}_{\omega} - \delta_{\theta,q} \hat{\sigma} \hat{\bs{x}}^{\top} \bs{L}_{\omega} \hat{\bs{y}}_{\omega,q} + \\
        \quad \sum\limits_{ j \in \Cs{J}^{+}_{\omega,q} } \delta_{\omega, j, q}^{+} \hat{\bs{x}}^{\top} \bs{l}_{\omega,j} - 
        \sum\limits_{ j \in \Cs{J}^{-}_{\omega,q} } \delta_{\omega, j, q}^{-} \hat{\bs{x}}^{\top} 
 \bs{l}_{\omega,j} \ge \rho_{\omega}, \; q \in \Cs{Q} \\ 
    \end{array}
    \right\rbrace.
\end{equation}

As $\Cs{W}_{\omega}$ is a cone, a normalizing constraint like $\sum_{q \in \Cs{Q}} \delta_{\theta, q}=1$, or $\sum_{q \in \Cs{Q}} \Big(\sum_{j \in \Cs{J}^{+}_{\omega,q} }  \linebreak \delta^{+}_{\omega,j, q} + \sum_{j \in \Cs{J}^{-}_{\omega,q} } \delta^{-}_{\omega,j, q} + \delta_{\theta, q} + \bs{\delta}_{\bs{x},q}^{\top} \bs{e} \Big)=1$ can be added to $\Cs{W}_{\omega}$ once solving the CGLP \eqref{eq: CGLP}. 
Note that if the normalizing constraint $\sum_{q \in \Cs{Q}} \delta_{\theta, q}=1$ is used, then the CGLP allows for those facets of $\conv{\Cs{T}_{\omega}}$ that have a positive coefficient $\alpha_{\theta,\omega}$ for $\theta_{\omega}$, as $\alpha_{\theta,\omega} \ge \max_{q \in \Cs{Q}} \delta_{\theta,q} >0$ by the constraints in $\Cs{W}_{\omega}$. A valid inequality is given by 
    $\theta_{\omega} + \Big[\frac{\bs{\alpha}_{\bs{x},\omega}}{\alpha_{\theta,\omega} } \Big]^{\top} \bs{x} +\Big[\frac{\bs{\alpha}_{\bs{\lambda},\omega}}{\alpha_{\theta,\omega} } \Big]^{\top} \bs{\lambda}   \ge \frac{\rho_{\omega}}{\alpha_{\theta,\omega}}$. 

\section{Proofs of Section \ref{sec: solution}}

\proof{Proof of Proposition \ref{prop: Lagrangian}}
Let $(\hat{\bs{\pi}}_{\bs{y}}, \hat{\pi}_{\gamma}, \hat{\pi}_{\mu})$ be optimal slack variables
to \eqref{eq: Lagrangian}. 
Hence, for any $(\bs{x},\bs{\lambda},\bs{y}, \mu, \gamma)$, we have 
\begin{align*}
    & \Fs{L}_{\omega}(\bs{x},\bs{\lambda},\bs{y}, \mu, \gamma, \hat{\sigma}, \hat{\bs{\pi}}, \hat{\bs{\pi}}_{\bs{y}}, \hat{\pi}_{\gamma}, \hat{\pi}_{\mu})  = - \hat{\sigma}  \bs{\lambda}^{\top} \bs{g}(\bs{\xi}_{\omega})+ \hat{\bs{\pi}}^{\top} (\bs{B}_{\omega}\bs{x}+\bs{b}_{\omega}) + \gamma ( \ol{p}_{\omega} -  \hat{\pi}_{\gamma} - \hat{\sigma} ) - \mu (\ul{p}_{\omega} +\hat{\pi}_{\mu} - \hat{\sigma})   \\
    & \qquad + (\hat{\sigma} \bs{q}_{\omega}^{\top} - \hat{\bs{\pi}}_{\bs{y}}^{\top} - \hat{\bs{\pi}}^{\top} \bs{D}_{\omega} + \hat{\sigma}  \bs{x}^{\top} \bs{L}_{\omega})\bs{y} =  - \hat{\sigma}  \bs{\lambda}^{\top} \bs{g}(\bs{\xi}_{\omega})  + \hat{\bs{\pi}}^{\top} (\bs{B}_{\omega}\bs{x}+\bs{b}_{\omega}) + \hat{\sigma}  (\bs{x} - \hat{\bs{x}})^{\top} \bs{L}_{\omega} \bs{y}, 
\end{align*}
where the second equality follows from the facts that $\ol{p}_{\omega} -  \hat{\pi}_{\gamma} - \hat{\sigma}=0$, $\ul{p}_{\omega} +\hat{\pi}_{\mu} - \hat{\sigma}=0$, and  $\hat{\sigma} \bs{q}_{\omega}^{\top} - \hat{\bs{\pi}}_{\bs{y}}^{\top} - \hat{\bs{\pi}}^{\top} \bs{D}_{\omega} + \hat{\sigma} \hat{\bs{x}}^{\top} \bs{L}_{\omega}=\bs{0}^{\top}$. 
Now, given that $Q_{\omega}(\bs{x},\bs{\lambda}, \sigma, \bs{\pi}, \bs{\pi}_{\bs{y}}, \pi_{\gamma}, \pi_{\mu}):=\min_{\bs{y}, \mu, \gamma} \ \Fs{L}_{\omega}(\bs{x},\bs{\lambda},\bs{y}, \mu, \gamma, \sigma, \bs{\pi}, \bs{\pi}_{\bs{y}}, \pi_{\gamma}, \pi_{\mu})$,  \eqref{eq: Lagrangian-nonconvex} follows, which also proves the last statement. 
\hfill \Halmos
\endproof

\proof{Proof of Proposition \ref{lem: optimality_conditions}}
Given that $\hat{\sigma} \ge 0$, Proposition \ref{prop: Lagrangian} implies that the condition 
    $(\bs{x} - \hat{\bs{x}})^{\top} \bs{L}_{\omega} (\bs{y}_{\omega} - \bs{y}^*_{\omega}) \ge 0$ for all $\bs{y}_{\omega}$ 
is sufficient and necessary for the optimality of $\bs{y}^*_{\omega}$. 
Because $\hat{\sigma}  (\bs{x} - \hat{\bs{x}})^{\top} \bs{L}_{\omega} \bs{y}_{\omega}$ is linear in $\bs{y}$, an optimal $\bs{y}^*_{\omega}$ is a boundary point of the space of $\bs{y}_{\omega}$, as stated in the lemma.  
\hfill \Halmos
\endproof

Given that $\Fs{L}_{\omega}(\bs{x},\bs{\lambda}, \bs{y}, \mu, \gamma, \hat{\sigma}, \hat{\bs{\pi}}, \hat{\bs{\pi}}_{\bs{y}}, \hat{\pi}_{\gamma}, \hat{\pi}_{\mu})$ only depends on primal variables $\bs{y}$ and dual multipliers  $(\hat{\sigma},\hat{\bs{\pi}})$, we suppress $(\mu, \gamma)$ and $(\hat{\bs{\pi}}_{\bs{y}}, \hat{\pi}_{\gamma}, \hat{\pi}_{\mu})$ from its  arguments. Moreover, $\Fs{L}_{\omega}(\cdot)$ depends on $\hat{\bs{x}}$ only through $(\hat{\sigma},\hat{\bs{\pi}})$; hence, we add $\hat{\bs{x}}$ to its arguments, i.e.,  $\Fs{L}_{\omega}(\bs{x},\bs{\lambda}, \bs{y}, \hat{\sigma}, \hat{\bs{\pi}}; \hat{\bs{x}})$, when needed for clarity. Similarly, we write $Q_{\omega}(\bs{x},\bs{\lambda}, \hat{\sigma}, \hat{\bs{\pi}}; \hat{\bs{x}})$.

The proof of Theorem~\ref{thm: DD-finite-convergence-nonconvex} relies on the following technical lemma. 
Consider $\Cs{L}$, as defined in Lemma~\ref{lem: boundedness_dual}. 
For $\omega \in [N]$ and a fixed $(\bs{x},\bs{\lambda}) \in \Cs{X} \times \Cs{L}$, let $\Cs{F}_{\omega}(\bs{x},\bs{\lambda})$ and $\Cs{E}_{\omega}(\bs{x},\bs{\lambda})$ denote the (primal) feasible region of problem \eqref{eq: Q_k_nonconvex-recourse} and  
the (dual) feasible region of problem \eqref{eq: Lagrangian}, respectively. Moreover, let $\Cs{D}_{\omega}(\bs{x},\bs{\lambda})$  denote the 
set of optimal solutions for problem \eqref{eq: Lagrangian}. 
Recall that a set-valued function $t \mapsto F(t): \Cs{D} \rightrightarrows \Bs{R}$ is upper semicontinuous (u.s.c.) at $t \in \Cs{D}$ if $\lim_{v \rightarrow \infty} t_{v} = t$, $y_{v} \in F(t_{v})$, and $\lim_{v \rightarrow \infty} y_{v} = y$ imply that $y \in F(t)$. 
A real-valued function $t \mapsto F(t): \Cs{D} \rightarrow \Bs{R}$ is lower semicontinuous (l.s.c.) at $t \in \Cs{D}$ if $\lim_{v \rightarrow \infty} t_{v} = t$ implies that $\liminf_{v \rightarrow \infty} F(t_{v}) = F(t)$. 

\begin{lemma}
\label{lem: G_properties}
For $\omega \in [N]$, we have: 
\begin{enumerate}[label=\roman*.]
    \item \label{i} $\Cs{D}_{\omega}(\bs{x},\bs{\lambda})$ is compact for a fixed $(\bs{x},\bs{\lambda}) \in \Cs{X} \times \Cs{L}$, 
    \item \label{F} $\Cs{F}_{\omega}(\bs{x},\bs{\lambda})$ is compact for a fixed $(\bs{x},\bs{\lambda}) \in \Cs{X} \times \Cs{L}$, 
    \item \label{E} The set-valued function $\Cs{E}_{\omega}(\bs{x},\bs{\lambda})$ is continuous in $(\bs{x},\bs{\lambda})$ on $\Cs{X} \times \Cs{L}$, 
    \item \label{G} The real-valued function $G_{\omega}(\bs{x},\bs{\lambda})$, as defined in \eqref{eq: Q_k_nonconvex-recourse}, is l.s.c. in $(\bs{x},\bs{\lambda})$ on $\Cs{X} \times \Cs{L}$, and 
    \item \label{D} The set-valued function  $\Cs{D}_{\omega}(\bs{x},\bs{\lambda})$ is u.s.c. in $(\bs{x},\bs{\lambda})$ on $\Cs{X} \times \Cs{L}$. 
\end{enumerate}
\end{lemma}

\proof{Proof of Lemma \ref{lem: G_properties}}
Consider a fixed $(\bs{x},\bs{\lambda}) \in \Cs{X} \times \Cs{L}$. Recall that $G_{\omega}(\bs{x},\bs{\lambda})$, defined in \eqref{eq: Q_k}, is well defined. Thus, $\Cs{F}_{\omega}(\bs{x},\bs{\lambda})$ is nonempty. 
Consequently, by a similar argument as that in Remark \ref{rem: bounded_dual}, $\Cs{D}_{\omega}(\bs{x},\bs{\lambda})$ is a bounded (and closed) set. 
Now, note that for a fixed $(\bs{x},\bs{\lambda}) \in \Cs{X} \times \Cs{L}$, $\Cs{F}_{\omega}(\bs{x},\bs{\lambda})$ is closed and bounded given that $\Cs{Y}(\bs{x},
\bs{\xi}_{\omega})$ is bounded by Assumption {\bf (A3)}, and the fact that $\mu$ and $\gamma$ are bounded by $|h(\bs{x},\bs{\xi}_{\omega})  - \bs{\lambda}^{\top} \bs{g}(\bs{\xi}_{\omega})|$ from above. By the boundedness of $\Cs{F}_{\omega}(\bs{x},\bs{\lambda})$ and  a direct application of \cite[Corollary~11]{wets1985continuity}, $\Cs{E}_{\omega}(\bs{x},\bs{\lambda})$ is continuous in $(\bs{x},\bs{\lambda})$ on $\Cs{X} \times \Cs{L}$. By the continuity of $\Cs{E}_{\omega}(\bs{x},\bs{\lambda})$ and the direct application of \cite[Theorem~2]{wets1985continuity}, we have that $G_{\omega}(\bs{x},\bs{\lambda})$ is l.s.c. in $(\bs{x},\bs{\lambda})$ on $\Cs{X} \times \Cs{L}$. Moreover, $\Cs{D}_{\omega}(\bs{x},\bs{\lambda})$ is equivalent to 
\begin{equation*}
\Set*{(\bs{\pi}, \sigma) \in \Cs{E}_{\omega}(\bs{x},\bs{\lambda})}{ G_{\omega}(\bs{x},\bs{\lambda})= -\sigma  \bs{\lambda}^{\top} \bs{g}(\bs{\xi}_{\omega})  + \bs{\pi}^{\top} (\bs{B}_{\omega}\bs{x}+\bs{b}_{\omega}) }.
\end{equation*}
Since the objective function of \eqref{eq: Lagrangian}, $-\sigma  \bs{\lambda}^{\top} \bs{g}(\bs{\xi}_{\omega})  + \bs{\pi}^{\top} (\bs{B}_{\omega}\bs{x}+\bs{b}_{\omega})$, is continuous on $\Cs{E}_{\omega}(\bs{x},\bs{\lambda}) \times \Cs{X} \times \Cs{L}$ and the set-valued function $\Cs{E}_{\omega}(\bs{x},\bs{\lambda})$ is continuous on $\Cs{X} \times \Cs{L}$,
the direct application of \cite[Theorem~1.5]{meyer1970validity} shows the set-valued function $\Cs{D}_{\omega}(\bs{x},\bs{\lambda})$ is u.s.c. in $(\bs{x},\bs{\lambda})$ on $\Cs{X} \times \Cs{L}$. \hfill \Halmos
\endproof

\proof{Proof of Theorem \ref{thm: DD-finite-convergence-nonconvex}}
    To prove the finite convergence of Algorithm \ref{alg: DD-nonconvex-recourse}, we need to show that the ``while" loop terminates in a finite number of iterations, generating an $\epsilon$-optimal solution to \eqref{eq: DRO_Obj}. 

    By contradiction, suppose that the ``while" loop does not terminate in a finite number of iterations. Let $\{(\bs{x}^{t}, \bs{\lambda}^{t}, \bs{\theta}^{t})\}$ be the sequence of iterates generated in Line \ref{line: iterates} and $\{(\bs{\pi}_{\omega}^{t}, \sigma_{\omega}^{t}) \in \Cs{D}_{\omega}(\bs{x}^t,\bs{\lambda}^t)\}$, $\omega \in [N]$, be the sequence of dual multipliers  generated in Line \ref{line: duals}. 
    We show that  $\theta_{\omega}^{t} \ge G_{\omega}(\bs{x}^{t}, \bs{\lambda}^{t}) - \frac{\epsilon}{2}$, $\omega \in [N]$, for all sufficiently large $t$. 
    This implies that we have 
        $\mathrm{LB} - \bs{c}^{\top} \bs{x}^t - (\bs{x}^t)^{\top} \bs{C} \bs{x}^t =  \bs{\lambda}^{t} \bs{\vartheta}(\bs{x}^{t})
		+ \sum_{\omega \in [N] } \theta_{\omega}^{t} \ge \bs{\lambda}^{t} \bs{\vartheta}(\bs{x}^{t}) 
		+ \sum_{\omega \in [N]} G_{\omega}(\bs{x}^{t},\lambda^{t}) - \frac{\epsilon}{2} \ge \mathrm{UB} - \bs{c}^{\top} \bs{x}^t - (\bs{x}^t)^{\top} \bs{C} \bs{x}^t  -\frac{\epsilon}{2}$,
    contradicting that the ``while'' loop does not terminate in a finite number of iterations. 
    Given Assumption {\bf (A4)}, we have $\mathrm{LB} \le z^{t} + \frac{\epsilon}{2}$.  Thus, Algorithm \ref{alg: DD-nonconvex-recourse} returns the $\epsilon$-optimal value $\textrm{UB}$ because $\textrm{UB} \le z^{t} + \epsilon$, with the corresponding $\epsilon$-optimal solution $(\bs{x}^{t}, \bs{\lambda}^{t})$.

    Now, we prove that $\theta_{\omega}^{t} \ge G_{\omega}(\bs{x}^{t}, \bs{\lambda}^{t}) - \frac{\epsilon}{2}$, $\omega \in [N]$, for all sufficiently large $t$. Note that $\{\bs{x}^{t}\}$ and $\{\bs{\lambda}^{t}\}$ are bounded by the compactness of $\Cs{X}$ and $\Cs{L}$ (by Lemma \ref{lem: boundedness_dual}), respectively. Moreover, $\{\bs{\theta}^{t}\}$ is a nondecreasing bounded sequence from below given that $\theta_{\omega}^t$ is an underestimator of $G_{\omega}(\bs{x},\bs{\lambda})$, $\omega \in [N]$, and accumulation of constraints in $\Cs{S}^t$. In addition, $\{\bs{\theta}^{t}\}$ is bounded from above given that  $\theta_{\omega}^t \le G_{\omega}(\bs{x}^t,\bs{\lambda}^t)$, $\omega \in [N]$.  Consequently, there is a convergent subsequence, say $\Cs{K}$, $\{(\bs{x}^{t}, \bs{\lambda}^{t}, \bs{\theta}^{t})\}_{t \in \Cs{K}}$. In addition, the associated sequence of optimal dual multipliers $\{(\bs{\pi}_{\omega}^{t}, \sigma_{\omega}^{t})\}_{t \in \Cs{K}}$, $\omega \in [N]$, is bounded by Lemma \ref{lem: G_properties}.\ref{i}; hence, there is a convergent subsequence on $\Cs{K}^{\prime} \subseteq \Cs{K}$. 
    Let $(\overline{\bs{x}}, \overline{\bs{\lambda}}, \overline{\bs{\theta}}) \in \Cs{X} \times \Cs{L}$ (by closedness) and $(\overline{\bs{\pi}}_{\omega}, \overline{\sigma}_{\omega})$, $\omega \in [N]$,  be a limit point of $\{(\bs{x}^{t}, \bs{\lambda}^{t}, \bs{\theta}^{t})\}$ and $\{(\bs{\pi}_{\omega}^{t}, \sigma_{\omega}^{t})\}$, $\omega \in [N]$, on $\Cs{K}^{\prime}$, respectively. As $\Cs{D}_{\omega}(\bs{x},\bs{\lambda})$, $\omega \in [N]$, is u.s.c. at  
    $(\overline{\bs{x}}, \overline{\bs{\lambda}})$ by Lemma \ref{lem: G_properties}.\ref{D}, we have that $(\overline{\bs{\pi}}_{\omega}, \overline{\sigma}_{\omega}) \in \Cs{D}_{\omega}(\overline{\bs{x}},\overline{\bs{\lambda}})$, $\omega \in [N]$. 
    Let $\{(\bs{y}_{\omega}^{t}, \mu^t_{\omega}, \gamma^t_{\omega}) \in \Cs{F}_{\omega}(\bs{x}^{t},\bs{\lambda}^{t}) \}$, $\omega \in [N]$, be the sequence of associated optimal primal solutions to \eqref{eq: Q_k_nonconvex-recourse}. Given that $\Cs{F}_{\omega}(\bs{x}^{t},\bs{\lambda}^{t})$ is compact by Lemma \ref{lem: G_properties}.\ref{F}, there is a convergent subsequence on $\Cs{K}^{\prime\prime} \subseteq \Cs{K}^{\prime}$ with a limit point $(\overline{\bs{y}}_{\omega}, \overline{\mu}_{\omega}, \overline{\gamma}_{\omega})$, $\omega \in [N]$,   on $\Cs{K}^{\prime\prime}$. 
    
    We now claim that  $\Fs{L}_{\omega}(\overline{\bs{x}},\overline{\bs{\lambda}},\bs{y}_{\omega,q},  \overline{\sigma}_{\omega},\overline{\bs{\pi}}_{\omega}; \overline{\bs{x}})= \Fs{L}_{\omega}(\overline{\bs{x}},\overline{\bs{\lambda}},\overline{\bs{y}}_{\omega},  \overline{\sigma}_{\omega},\overline{\bs{\pi}}_{\omega}; \overline{\bs{x}})$ for every $q \in \Cs{Q}$. 
    Note that for every $(\bs{x}, \bs{\lambda}) \in \Cs{X} \times \Cs{L}$, we have $\lim_{t \rightarrow \infty} \Fs{L}_{\omega}(\bs{x}, \bs{\lambda}, \bs{y}_{\omega,q}, \sigma^{t}_{\omega}, \bs{\pi}^{t}_{\omega}; \bs{x}^t) \linebreak  =- \overline{\sigma}_{\omega}  \bs{\lambda}^{\top} \bs{g}(\bs{\xi}_{\omega})  + \overline{\bs{\pi}}^{\top}_{\omega} (\bs{B}_{\omega}\bs{x}+\bs{b}_{\omega}) + \overline{\sigma}_{\omega}  (\bs{x} - \overline{\bs{x}})^{\top} \bs{L}_{\omega} \bs{y}_{\omega,q} \! = \! \Fs{L}_{\omega}(\bs{x},\bs{\lambda},\bs{y}_{\omega,q},  \overline{\sigma}_{\omega},\overline{\bs{\pi}}_{\omega}; \overline{\bs{x}})$ by the continuity of $\Fs{L}_{\omega}(\bs{x}, \bs{\lambda}, \bs{y}_{\omega,q}, \hat{\sigma}, \hat{\bs{\pi}}; \hat{\bs{x}})$ at  $(\hat{\sigma}, \hat{\bs{\pi}}; \hat{\bs{x}})$. Thus, $\Fs{L}_{\omega}(\overline{\bs{x}},\overline{\bs{\lambda}},\bs{y}_{\omega,q},  \overline{\sigma}_{\omega},\overline{\bs{\pi}}_{\omega}; \overline{\bs{x}})   \linebreak =-\overline{\sigma}_{\omega}  \overline{\bs{\lambda}}^{\top} \bs{g}(\bs{\xi}_{\omega}) + \overline{\bs{\pi}}_{\omega}^{\top} (\bs{B}_{\omega} \overline{\bs{x}}+\bs{b}_{\omega}) = \Fs{L}_{\omega}(\overline{\bs{x}},\overline{\bs{\lambda}},\overline{\bs{y}}_{\omega},  \overline{\sigma}_{\omega},\overline{\bs{\pi}}_{\omega}; \overline{\bs{x}})$ for every $q \in \Cs{Q}$.  

    Given the validity of optimality cut generated at Line \ref{line: cut} for the corresponding set $\conv{T_{\omega}}$, we have $\theta_{\omega} \ge \sigma^{t}_{\omega}   
         - \bs{\lambda}^{\top} \bs{g}(\bs{\xi}_{\omega}) +  (\bs{B}_{\omega}\bs{x}+\bs{b}_{\omega})^{\top} \bs{\pi}^{t}_{\omega}  + \sigma^t_{\omega} (\bs{x} - \bs{x}^{t})^{\top} \bs{L}_{\omega}\bs{y}_{\omega,q}$ for some $q \in \Cs{Q}$. Moreover, given the accumulation of cuts, we have 
         $\theta_{\omega}^{t+1} \ge -\sigma^{t}_{\omega}   
         \bs{g}(\bs{\xi}_{\omega})^{\top} \bs{\lambda}^{t+1}  +  (\bs{B}_{\omega}\bs{x}^{t+1}+\bs{b}_{\omega})^{\top} \bs{\pi}^{t}_{\omega}  + \sigma^t_{\omega}  (\bs{x}^{t+1} - \bs{x}^{t})^{\top} \bs{L}_{\omega}\bs{y}_{\omega,q}$. Thus, taking the limit on $\Cs{K}^{\prime\prime}$, we have 
         $\overline{\theta}_{\omega}
         \ge \Fs{L}_{\omega}(\overline{\bs{x}},\overline{\bs{\lambda}},\bs{y}_{\omega,q},  \overline{\sigma}_{\omega},\overline{\bs{\pi}}_{\omega}; \overline{\bs{x}})$. And, using the above claim yields $\overline{\theta}_{\omega} \ge \Fs{L}_{\omega}(\overline{\bs{x}},\overline{\bs{\lambda}},\overline{\bs{y}}_{\omega},  \overline{\sigma}_{\omega},\overline{\bs{\pi}}_{\omega}; \overline{\bs{x}})$, $\omega \in [N]$. 
         Now, given that $(\overline{\bs{\pi}}_{\omega}, \overline{\sigma}_{\omega}) \in \Cs{D}_{\omega}(\overline{\bs{x}},\overline{\bs{\lambda}})$, $\omega \in [N]$, by strong duality we have $G_{\omega}(\overline{\bs{x}},\overline{\bs{\lambda}})=\Fs{L}_{\omega}(\overline{\bs{x}},\overline{\bs{\lambda}},\overline{\bs{y}}_{\omega},  \overline{\sigma}_{\omega},\overline{\bs{\pi}}_{\omega}; \overline{\bs{x}})$. Hence, $\overline{\theta}_{\omega}  \ge G_{\omega}(\overline{\bs{x}},\overline{\bs{\lambda}})$. Finally, because $G_{\omega}(\bs{x},\bs{\lambda})$ is l.s.c. by Lemma \ref{lem: G_properties}.\ref{G}, we have $\theta_{\omega}^{t} \ge G_{\omega}(\bs{x}^{t}, \bs{\lambda}^{t}) - \frac{\epsilon}{2}$, $\omega \in [N]$, for all sufficiently large $t$. This completes the proof. \hfill \Halmos  
\endproof

\section{Computational Performance Profile for the Multiproduct Newsvendor Problem}
\label{sec: EC_NV}

For each configuration of parameters, we compared the solution times from the two approaches, \texttt{DECOMPOSED} and \texttt{DEF}, to solve \texttt{DD} and \texttt{DI}, on the subset of instances solved within 3600 seconds.
Let $T_I(A)$ denote the computational time in seconds for solving instance $I$ with method
$A \in \{\texttt{DD-DEC}, \texttt{DD-DEF},\texttt{DI-DEC}, \texttt{DI-DEF}\}$, and let $\eta_I(A)$ be a relative metric calculated as:
\begin{equation*}
\eta_I(A)
=
\frac{T_I(A)-\min_A T_I(A)}
{\max_A T_I(A)-\min_A T_I(A)}
\in [0,1].    
\end{equation*}
Any point $(\Omega_T,\beta_T)$ on the computational time performance profile for method $A$ indicates that $\eta_I(A)$ was at most $\Omega_T$ for a fraction $\beta_T$ of the solved instances. Thus, the point $(0,\beta_T)$ implies that $\beta_T$ fraction of instances were solved quickest by $A$.
Fig.~\ref{fig: NV_profile} shows the computational time performance profiles for the newsvendor problem. 

\begin{figure}[!tb]
	\centering
    \includegraphics[width=0.5\linewidth]{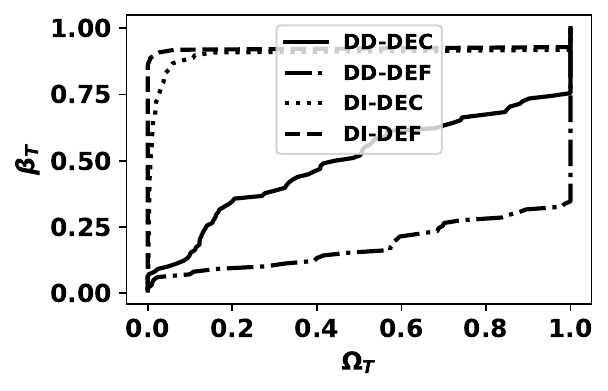}
	\caption{\label{fig: NV_profile} Computational time performance profiles for the newsvendor problem with a price-dependent, $\texttt{DD}$, and decision-independent, $\texttt{DI}$, ambiguity set, with $\gamma=1$, $\varsigma=1$, $\rho_{\mu}=\rho_{\sigma}=1$, and $\tau_\mu=\underline{\tau}_\sigma=0$.}
\end{figure}	

Consistent with Table~\ref{T: NV_Comp_A}, the profiles show that the price-independent models are substantially easier to solve than their price-dependent counterparts. Both \texttt{DI-DEC} and \texttt{DI-DEF} rise sharply near $\Omega_T=0$, indicating that one of the two decision-independent approaches is the fastest method for a large fraction of instances; this agrees with the observation in the main text that both decision-independent formulations were solved within sixty seconds on average. In contrast, the two \texttt{DD} curves are shifted to the right, reflecting the additional computational burden created by price-dependent ambiguity.
Among the two methods for solving the price-dependent model, \texttt{DD-DEC} dominates \texttt{DD-DEF} over most of the profile. This behavior is also reflected in Table~\ref{T: NV_Comp_A}: \texttt{DECOMPOSED} solved all \texttt{DD} instances to optimality within the time limit, whereas \texttt{DEF} terminated with nonzero optimality gaps or failed to solve several larger instances, especially as $n$ and $N$ increased. Hence, while \texttt{DEF} can be competitive on some smaller or less conservative instances, the decomposition algorithm is the more reliable approach for the challenging price-dependent newsvendor instances.

\section{Numerical Results for the Facility Location Problem}

\label{sec: FL}
Consider the facility location problem introduced in Section \ref{sec: FL_app}. Note that $h(\bs{x},\bs{\xi})$ is in the form of \eqref{eq: recourse} with an objective function in the form of \eqref{eq: nonconvex-recourse}, with $\bs{L}=\bs{0}$, yielding a convex recourse function. Hence, an L-shaped-type algorithm is applicable to solve this problem. We assume $g_j > r_j > c_{ij}$ for $i \in [n]$ and $j \in [m]$. 
Also, in light of Lemma \ref{lem: boundedness_dual}, we have 
$\Cs{L}=\{\bs{\lambda} : 0 \le \lambda_\iota \le \overline{g}\overline{\xi}, \; \iota \in [s] \}$, where 
$\overline{g}=\max_{ j \in [m]} g_j$ and $\overline{\xi}=\max_{ j \in [m], \omega \in [N]} \xi_j^{\omega}$. 

\subsection{Experiment Design}

To conduct experiments, we generated spatial instances using clustered and random schemes. Given $n$ potential facility locations and $m$ customers, we set the  number of clusters to $K:=\min\{m,\max\{2,\min(\lfloor m/10 \rfloor,\lfloor \lceil n/2 \rceil/3 \rfloor)\}\}$. We then sampled the $K$ cluster centers uniformly from $[15,85]^2$ subject to pairwise Euclidean distance at least $35$. The $m$ customers were split across clusters as evenly as possible, so each cluster contains either $\lfloor m/K \rfloor$ or $\lceil m/K \rceil$ customers, and their coordinates were obtained by adding Gaussian noise with standard deviation $5$ to the corresponding cluster center. Moreover, one candidate facility location was placed near each cluster center using Gaussian noise with standard deviation $3$, and the remaining $n-K$ candidate locations were sampled uniformly from $[0,100]^2$. All coordinates were clipped to $[0,100]^2$. In the random scheme, all coordinates were chosen in $[0,100]^2$.
Moreover, we chose $C_i \sim \text{Uniform}(10,20)$, $i \in [n]$. For  $j \in [m]$, we set $g_j=225$ and $r_j=140$.

We let $c_{ij}$ be the Euclidean distance between location $i \in [n]$ and customer $j \in [m]$. 
For each location $i \in [n]$, we computed its average distance to all customers and normalized these averages across $i \in [n]$ to obtain $\delta_i \in [0,1]$. We then defined a location-specific attractiveness factor $a_i:=1+2\delta_i$. For each customer $j \in [m]$, we formed $w_{ij}:=a_i\exp(-c_{ij}/\varsigma)$, $i\in[n]$, normalized these weights over $i \in [n]$ to obtain $\tilde{w}^{\mu}_{ij}$. We set $u^\mu_{ij}=\rho_{\mu} \tilde{w}^{\mu}_{ij}$.
Next, we calculated $\tilde{w}^\sigma_{ij}=\tilde{w}^\mu_{ij}-\frac{1}{2}\min_{k \in [n]} \tilde{w}^\mu_{kj}$, $i\in[n]$. We then set $u^\sigma_{ij}=\rho_{\sigma} \tilde{w}^{\sigma}_{ij}$. Parameters $\varsigma$, $\rho_{\mu}$, and $\rho_{\sigma}$ are tunable and control the {\it degree} of decision-dependency, where $\rho_{\sigma} \le 1$ so that $\sum_{i \in [n]}u^\sigma_{ij}<1$. 
A higher $\rho_{\mu}$ and $\rho_{\sigma}$ imply a stronger decision-dependency, with a greater demand mean and a smaller demand variance, respectively. 

By construction, the distance-decay term $\exp(-c_{ij}/\varsigma)$ makes nearby facilities both cheaper to provide service from and more likely to attract demand. If this were the only driver of demand, then the decision-dependent and decision-independent models would tend to favor the same locations near dense customer clusters, leaving little contrast between them. The attractiveness factor $a_i$ is therefore introduced to create tension between transportation efficiency and demand inducement: locations that are relatively farther from customers can still generate higher demand and revenue because of their intrinsic attractiveness. This separation makes the value of modeling decision dependence more visible, since the decision-dependent model can exploit attraction-driven demand shifts, whereas the decision-independent model continues to respond only to transportation cost and capacity considerations.

For each training set, we chose the vector of nominal mean of the random demand $\bs{\xi}$, $\overline{\bs{\mu}}$, independently from  $\text{Uniform}(20,40)$ and set $\overline{\bs{\sigma}}=\gamma \overline{\bs{\mu}}$, where $\gamma$ is a tunable parameter to control the coefficient of variation.
Realizations of the random demand $\xi_j$, $j \in [m]$, were generated independently from a folded normal distribution with mean $\overline{\mu}_j$ and standard deviation $\overline{\sigma}_j$.  

\subsection{Computational Results}

In this section, we compare the computational performance of the L-shaped-type algorithm, denoted as \texttt{DECOMPOSED}, against the MINLP deterministic equivalent formulation, presented in \eqref{eq: reform-simple-moment-recourse-linearized}, which is solved using an off-the-shelf nonconvex solver and denoted as \texttt{DEF}.
To do so, we set a time limit of 14400 seconds and followed the random scheme to generate coordinates.  
We present the computational results for $n \in \{10,15,20\}$, $m=2n$, $d=\lceil n/2 \rceil$, hyperparameters $\tau_\mu=0.4$, $\overline{\tau}_\sigma=1.4$, and $\underline{\tau}_\sigma=0.6$ for the ambiguity set, $\varsigma \in \{10,20\}$ and $\rho_{\mu}=\rho_{\sigma} =1$ for the degree of decision dependency, $\gamma=1$ for the coefficient of variation, and $N \in \{100, 200,300,400,500\}$. 

\begin{table}
\centering
\scriptsize
\caption{Comparison of \texttt{DEF} and \texttt{DECOMPOSED} to solve the facility location problem with a decision-dependent, $\texttt{DD}$, and decision-independent, $\texttt{DI}$, ambiguity set, with $\gamma=1$, $\rho_{\mu}=\rho_{\sigma}=1$, $\tau_\mu=0.4$, $\underline{\tau}_\sigma=0.6$, and $\overline{\tau}_\sigma=1.4$.}
\label{T: FL_Comp_A}
 \begin{threeparttable}
\begin{tabular}{lllllllllll}
\toprule
 &  & & \multicolumn{4}{c}{$\texttt{DD}$} & \multicolumn{4}{c}{$\texttt{DI}$} \\
 \cmidrule(lr){4-7} \cmidrule(lr){8-11} 
  & & & \multicolumn{2}{c}{\texttt{DECOMPOSED}} & \multicolumn{2}{c}{\texttt{DEF}} & \multicolumn{2}{c} {\texttt{DECOMPOSED}} & \multicolumn{2}{c}{\texttt{DEF}} \\
 \cmidrule(lr){4-5} \cmidrule(lr){6-7}  \cmidrule(lr){8-9} \cmidrule(lr){10-11} 
$(n,m,d)$ & $\varsigma$ & $N$ &  Gap (\%) & Time (s) & Gap (\%) & Time (s) & Gap (\%) & Time (s) & Gap (\%) & Time (s) \\
\midrule
\multirow[t]{10}{*}{(10, 20, 5)} & \multirow[t]{5}{*}{10} & 100 & 0.0 & 47.44 & 0.0 & 24.36 & 0.0 & 14.81 & 0.0 & 0.68 \\
 &  & 200 & 0.0 & 99.4 & 0.0 & 56.59 & 0.0 & 37.17 & 0.0 & 1.84 \\
 &  & 300 & 0.0 & 156.74 & 0.0 & 94.69 & 0.0 & 56.65 & 0.0 & 2.76 \\
 &  & 400 & 0.0 & 209.91 & 0.0 & 141.99 & 0.0 & 63.53 & 0.0 & 4.14 \\
 &  & 500 & 0.0 & 258.53 & 0.0 & 200.79 & 0.0 & 77.85 & 0.0 & 5.54 \\
\cline{2-11}
 & \multirow[t]{5}{*}{20} & 100 & 0.0 & 49.38 & 0.0 & 24.28 & 0.0 & 14.59 & 0.0 & 0.69 \\
 &  & 200 & 0.0 & 95.56 & 0.0 & 52.15 & 0.0 & 37.5 & 0.0 & 1.75 \\
 &  & 300 & 0.0 & 159.84 & 0.0 & 92.0 & 0.0 & 56.43 & 0.0 & 2.77 \\
 &  & 400 & 0.0 & 219.28 & 0.0 & 132.25 & 0.0 & 65.69 & 0.0 & 4.22 \\
 &  & 500 & 0.0 & 245.13 & 0.0 & 193.16 & 0.0 & 77.25 & 0.0 & 5.5 \\
\midrule
\multirow[t]{10}{*}{(15, 30, 8)} & \multirow[t]{5}{*}{10} & 100 & 0.0 & 233.27 & 0.0 & 139.31 & 0.0 & 64.78 & 0.0 & 2.63 \\
 &  & 200 & 0.0 & 592.17 & 0.0 & 321.69 & 0.0 & 101.94 & 0.0 & 6.47 \\
 &  & 300 & 0.0 & 781.49 & 0.0 & 616.63 & 0.0 & 141.58 & 0.0 & 12.21 \\
 &  & 400 & 0.0 & 1176.99 & 0.0 & 773.71 & 0.0 & 176.01 & 0.0 & 17.98 \\
 &  & 500 & 0.0 & 1338.63 & 0.0 & 1059.04 & 0.0 & 220.73 & 0.0 & 24.55 \\
\cline{2-11}
 & \multirow[t]{5}{*}{20} & 100 & 0.0 & 276.04 & 0.0 & 154.06 & 0.0 & 65.42 & 0.0 & 2.63 \\
 &  & 200 & 0.0 & 531.14 & 0.0 & 435.59 & 0.0 & 102.4 & 0.0 & 6.57 \\
 &  & 300 & 0.0 & 801.53 & 0.0 & 762.77 & 0.0 & 143.0 & 0.0 & 12.26 \\
 &  & 400 & 0.0 & 1018.04 & 0.0 & 987.82 & 0.0 & 176.69 & 0.0 & 18.06 \\
 &  & 500 & 0.0 & 1046.15 & 0.0 & 1436.55 & 0.0 & 224.26 & 0.0 & 25.23 \\
\midrule
\multirow[t]{10}{*}{(20, 40, 10)} & \multirow[t]{5}{*}{10} & 100 & 0.0 & 1451.64 & 0.0 & 3843.03 & 0.01 & 122.41 & 0.0 & 5.45 \\
 &  & 200 & 0.0 & 2512.09 & 0.0 (7.9) & 9883.27 (1) & 0.0 & 305.16 & 0.0 & 19.54 \\
 &  & 300 & 0.0 & 3178.64 & 0.0 (11.93) & 12379.35 (3) & 0.0 & 484.21 & 0.0 & 34.84 \\
 &  & 400 & 0.0 & 5127.42 & 0.0 (10.84) & 13083.85 (4) & 0.0 & 607.55 & 0.0 & 51.39 \\
 &  & 500 & 0.0 & 7540.46 & - (12.47) & - (5) & 0.0 & 764.98 & 0.0 & 69.27 \\
\cline{2-11}
 & \multirow[t]{5}{*}{20} & 100 & 0.0 & 1549.21 & 0.0 & 7861.81 & 0.01 & 121.52 & 0.0 & 5.54 \\
 &  & 200 & 0.0 & 2700.16 & 0.0 (7.51) & 11405.7 (2) & 0.0 & 309.01 & 0.0 & 19.17 \\
 &  & 300 & 0.0 & 4017.26 & - (13.82) & - (5) & 0.0 & 487.19 & 0.0 & 34.33 \\
 &  & 400 & 0.0 & 4979.49 & - (14.93) & - (5) & 0.0 & 603.5 & 0.0 & 51.14 \\
 &  & 500 & 0.0 & 7180.68 & - (16.56) & - (5) & 0.0 & 767.91 & 0.0 & 68.27 \\
\bottomrule
\end{tabular}

 \end{threeparttable}
\end{table}

Table \ref{T: FL_Comp_A} reports the average computational time (over five training sets) for both DROs with decision-dependent, $\texttt{DD}$, and decision-independent ambiguity set, $\texttt{DI}$.  
For a DRO with a location-dependent ambiguity set, observe from Table~\ref{T: FL_Comp_A} that $\texttt{DECOMPOSED}$ found an optimal solution within the time limit for all instances, whereas $\texttt{DEF}$ could not find an optimal solution for many instances within the time limit. These instances are typically larger problems in terms of the number of scenarios, $N$, the number of potential locations, $n$, and the number of installed facilities, $d$. 
In addition, for instances that could be solved optimally within the time limit with both approaches, $\texttt{DEF}$  often found an optimal solution with less computational effort. An increase in the number of scenarios, $N$, the number of potential locations, $n$, and the number of installed facilities, $d$, led to longer runtimes for both approaches. 

On the other hand, a DRO with a location-independent ambiguity set is typically an easier problem to solve (Table~\ref{T: FL_Comp_A}). In particular,  $\texttt{DECOMPOSED}$ to solve \texttt{DI} was approximately 2 to 12 times faster than using $\texttt{DECOMPOSED}$ to solve \texttt{DD}. 
This observation further highlights the increased computational burden associated with solving the location-dependent DRO problem compared to its location-independent counterpart.

Using the same relative time metric in Section~\ref{sec: EC_NV}, Fig.~\ref{fig: FL_profile} reports the computational time performance profiles for the facility location problem over instances considered in Table~\ref{T: FL_Comp_A}, with a time limit of 14400 seconds.

\begin{figure}[!tb]
	\centering
    \includegraphics[width=0.5\linewidth]{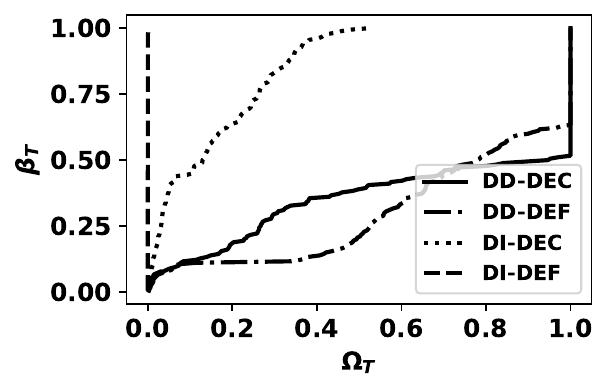}
	\caption{\label{fig: FL_profile} Computational time performance profiles for the facility location problem with a decision-dependent, $\texttt{DD}$, and decision-independent, $\texttt{DI}$, ambiguity set, with $\gamma=1$, $\rho_{\mu}=\rho_{\sigma}=1$, $\tau_\mu=0.4$, $\underline{\tau}_\sigma=0.6$, and $\overline{\tau}_\sigma=1.4$.}
\end{figure}

The profile reinforces the computational trends reported in Table~\ref{T: FL_Comp_A}. The \texttt{DI-DEF} curve rises almost immediately to one, showing that \texttt{DEF} is typically the fastest approach for \texttt{DI} instances. The \texttt{DI-DEC} curve also lies well above the corresponding decision-dependent curves for all values of $\Omega_T$, which is consistent with the observation that solving \texttt{DI} with \texttt{DECOMPOSED} is approximately 1.2 to 2.5 times faster than solving \texttt{DD} with \texttt{DECOMPOSED}. 
For \texttt{DD}, \texttt{DD-DEF} often performs favorably relative to \texttt{DD-DEC} on instances that it solves, as indicated by its steeper initial rise and by the smaller runtimes reported in Table~\ref{T: FL_Comp_A} for many small and medium instances. However, Table~\ref{T: FL_Comp_A} also shows that \texttt{DD-DEF} fails on several larger instances, whereas \texttt{DD-DEC} solves all instances within the time limit. Therefore, the \texttt{DEF} can be faster when it remains tractable, but \texttt{DECOMPOSED} provides the more dependable approach for larger \texttt{DD} instances.

\subsection{Out-of-Sample Performance} 
We also report the out-of-sample performance of solutions resulting from a DRO with decision-dependent, $\texttt{DD}$, and decision-independent ambiguity set, $\texttt{DI}$, for a problem instance with $n=10$, $m=20$, and $d=5$ on a clustered scheme, with $N=100$, 
$\varsigma=10$, $\rho_{\mu}, \rho_{\sigma} \in \{0.3,0.5,0.7\}$ for the degree of decision dependency, and $\gamma \in \{0.5,0.7\}$ 
for the coefficient of variation. 
To form the ambiguity sets, we used parameters $\tau_\mu$ and $\tau_\sigma$, where $\overline{\tau}_\sigma=1+\tau_\sigma$, and $\underline{\tau}_\sigma=1-\tau_\sigma$, with $\tau_\mu, \tau_\sigma \in \{0, 0.2\}$. 
Fig.~\ref{fig: FL_0.5_100} reports a two-sided 95\% confidence interval of these UCBs over 25 microsimulations, for $\gamma=0.5$ and $N=100$. 
These figures were generated very similarly to Fig.~\ref{fig: NV_oos_0.5_100}. 
Observe that the out-of-sample cost difference $(\texttt{DI} - \texttt{DD})$ is positive, indicating that $\texttt{DD}$ yielded a smaller out-of-sample performance than its \texttt{DI} counterpart. 
This pattern is consistent with the facility-location model: opening a facility changes both the demand it attracts and, through the ambiguity set, the way mean and dispersion are hedged.

\paragraph{Sensitivity with respect to $\rho_{\mu}$ and $\rho_{\sigma}$.}
Observe that at a fixed $\rho_{\sigma}$, the curves are typically upward sloping in $\rho_{\mu}$, especially when the coefficient of variation, $\gamma$, is moderate. This indicates that a larger $\rho_{\mu}$ typically leads to a larger positive gap $(\texttt{DI}-\texttt{DD})$. The reason is that a larger $\rho_{\mu}$ amplifies how facility openings relocate demand across customers, so \texttt{DI} model becomes increasingly misspecified on the mean side, whereas \texttt{DD} can adjust facility openings and capacity allocation to the demand pattern it endogenously induces.
For a fixed $\rho_{\mu}$, the ordering of the curves is always increasing in $\rho_{\sigma}$, so higher values of $\rho_{\sigma}$ leads to a larger positive gap $(\texttt{DI}-\texttt{DD})$. This suggests that accounting for location-specific dispersion can further improve the robustness of the \texttt{DD} solution. However, this $\rho_{\sigma}$ effect is more nuanced than the $\rho_{\mu}$ effect. Especially, when $\gamma$ is higher, some of the upper curves, corresponding to higher $\rho_{\sigma}$,  flatten or even decline as $\rho_{\mu}$ increases, indicating that once variance uncertainty and tail protection become dominant, both models are pushed toward similarly conservative facility configurations and the marginal value of modeling decision-dependent variance is reduced.

\paragraph{Sensitivity with respect to $\tau_\mu$ and $\tau_\sigma$.} 
Recall that $\tau_\mu$ governs the tolerance on the mean, whereas $\tau_\sigma$ governs the admissible band on the variance. A larger $\tau_\mu$ primarily enlarges the ambiguity set along directions that increase the {\it mean} of demand, and therefore magnifies the consequences of misspecifying decision-induced demand relocation. When those mean shifts are important, the slope of $(\texttt{DI}-\texttt{DD})$ with respect to $\rho_{\mu}$ becomes steeper, reflecting the fact that \texttt{DD} is better able to align facility openings with the demand pattern it endogenously creates.
By contrast, a larger $\tau_\sigma$ widens the \emph{spread} of demand and places more emphasis on hedging dispersion and tail outcomes. This tends to move both \texttt{DD} and \texttt{DI} toward more conservative solutions, which can compress the gap $(\texttt{DI}-\texttt{DD})$ because the problem becomes driven more by generic variance protection than by exploiting decision-dependent structure. This weakening effect is most visible in the higher $\gamma$ panels, where the curves are flatter and occasionally decrease.

Overall, the sensitivity analysis indicates that the value of incorporating decision dependence is largest when the ambiguity set is informative enough for location-specific mean and variance effects to matter, but not so wide that both models collapse to nearly the same conservative hedge.

\begin{figure}[!tb]
	\centering
    \includegraphics[width=0.5\linewidth]{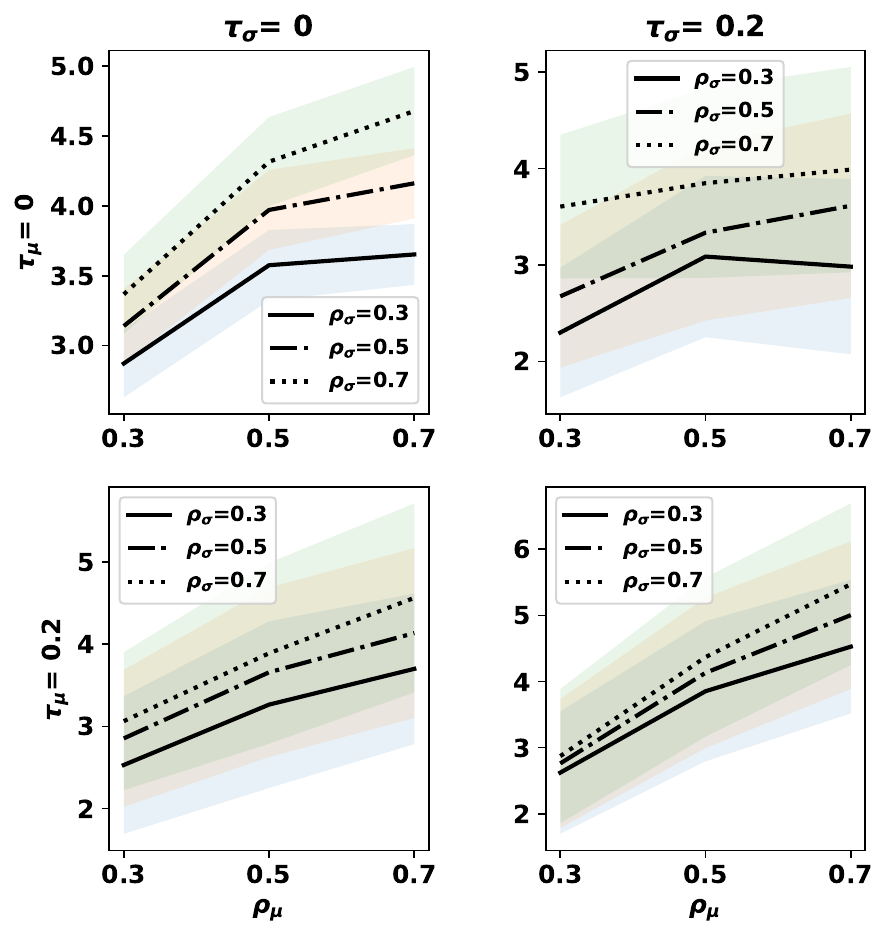}
	\caption{\label{fig: FL_0.5_100} 95\% UCB on the mean out-of-sample difference $(\texttt{DI} - \texttt{DD})$ for the  facility location problem with $n=10$, $m=20$, $d=5$, $\varsigma=10$, $\gamma=0.5$, and $N=100$.}
\end{figure}


\begin{figure}[!tb]
	\centering
    \includegraphics[width=0.5\linewidth]{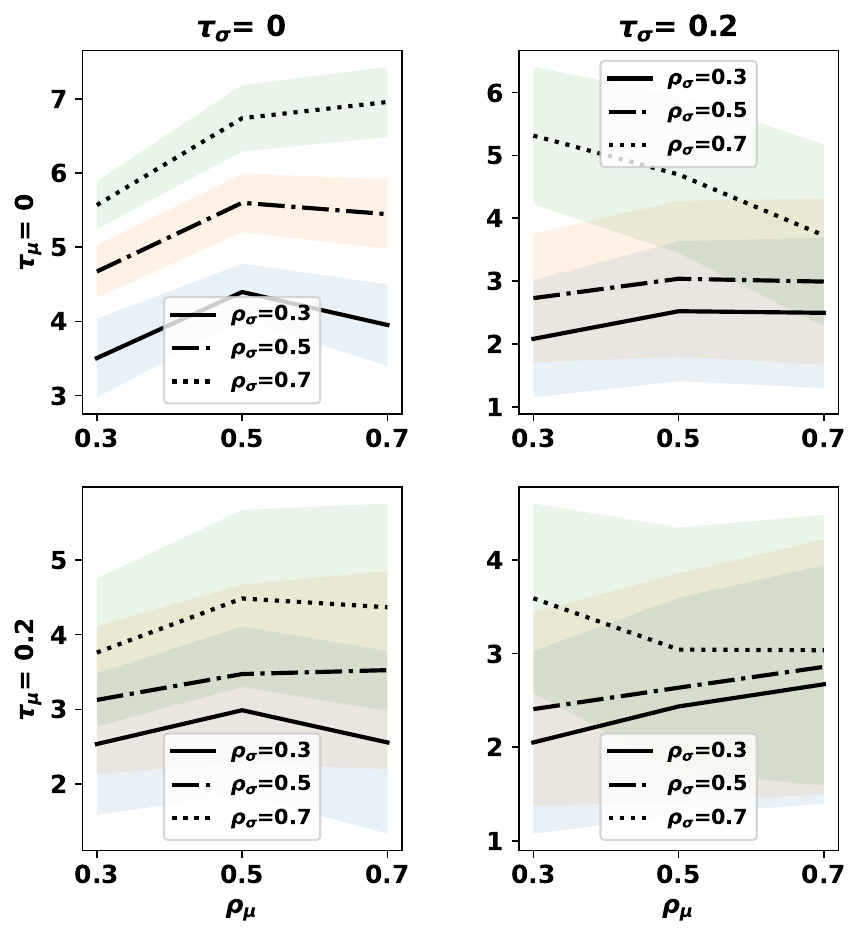}
	\caption{\label{fig: FL_0.7_100} 95\% UCB on the mean out-of-sample difference $(\texttt{DI} - \texttt{DD})$ for the  facility location problem with $n=10$, $m=20$, $d=5$, $\varsigma=10$, $\gamma=0.7$, and $N=100$.}
\end{figure}

\end{document}